\documentclass{article}
\usepackage{amsfonts}
\usepackage{amsthm}
\usepackage{amsmath}
\usepackage{amssymb}
\usepackage{tikz}
\usepackage{tikz-cd}
\usepackage{enumitem}
\usepackage{mathtools}
\usepackage{thmtools}
\usepackage[hypertexnames=false]{hyperref}
\usepackage[capitalize,nameinlink]{cleveref}
\usepackage[noadjust]{cite}

\usepackage{tocloft}
\usepackage[title]{appendix}
\usepackage[margin=1in]{geometry}

\theoremstyle{plain}
\newtheorem{theorem}{Theorem}[section]
\newtheorem{lemma}[theorem]{Lemma}
\newtheorem{corollary}[theorem]{Corollary}
\newtheorem{proposition}[theorem]{Proposition}

\theoremstyle{definition}
\newtheorem{definition}[theorem]{Definition}
\newtheorem{convention}[theorem]{Convention}
\newtheorem{construction}[theorem]{Construction}
\newtheorem{remark}[theorem]{Remark}
\newtheorem{example}[theorem]{Example}
\newtheorem{question}[theorem]{Question}

\let\Sec=\S

\newcommand{\R}{\mathbb{R}}

\newcommand{\N}{\mathbb{N}}
\newcommand{\Z}{\mathbb{Z}}
\newcommand{\D}{\mathbb{D}}
\newcommand{\A}{\mathbb{A}}

\newcommand{\F}{\mathbb{F}}
\newcommand{\T}{\mathbb{T}}
\renewcommand{\S}{\mathbb{S}}

\newcommand{\id}{\mathrm{id}}
\newcommand{\pt}{\mathrm{pt}}
\newcommand{\ex}{\mathrm{ex}}
\newcommand{\st}{\mathrm{st}}
\newcommand{\dual}{\mathrm{dual}}
\newcommand{\perf}{\mathrm{perf}}
\newcommand{\open}{\mathrm{open}}
\newcommand{\hyp}{\mathrm{hyp}}
\newcommand{\cont}{\mathrm{cont}}
\newcommand{\lax}{\mathrm{lax}}

\newcommand{\Space}{\mathcal{S}}
\newcommand{\Sp}{\mathrm{Sp}}
\newcommand{\Ab}{\mathrm{Ab}}
\newcommand{\PSh}{\mathrm{PShv}}
\newcommand{\Shv}{\mathrm{Shv}}
\newcommand{\CoShv}{\mathrm{CoShv}}
\newcommand{\LCH}{\mathrm{LCH}}
\newcommand{\Cat}{\mathrm{Cat}}

\newcommand{\PrL}{\mathrm{Pr}^{\mathrm{L}}}
\newcommand{\PrR}{\mathrm{Pr}^{\mathrm{R}}}
\newcommand{\PB}{\mathrm{PB}}
\newcommand{\CoSys}{\mathrm{CoSys}}
\newcommand{\Corr}{\mathrm{Corr}}
\newcommand{\BCFun}{\mathrm{BCFun}}

\newcommand{\SixFF}{\mathrm{6FF}}
\newcommand{\Calk}{\mathrm{Calk}}

\newcommand{\THH}{\mathrm{THH}}
\newcommand{\TC}{\mathrm{TC}}

\DeclareMathOperator{\Open}{Open}
\DeclareMathOperator{\CAlg}{CAlg}
\DeclareMathOperator{\Fun}{Fun}
\DeclareMathOperator{\Map}{Map}
\DeclareMathOperator{\Hom}{Hom}

\DeclareMathOperator{\Lan}{Lan}

\DeclareMathOperator{\Ind}{Ind}
\DeclareMathOperator{\Loc}{Loc}
\DeclareMathOperator{\op}{op}

\makeatletter
\def\colim{\qopname\relax m{colim}}
\makeatother 

\begin{document}

\title{Continuous six-functor formalism on locally compact Hausdorff spaces}
\author{Qingchong Zhu}
\date{}

\maketitle

\begin{abstract}
  We show that the functor sending a locally compact Hausdorff space $X$ to the $\infty$-category of spectral sheaves $\Shv(X; \Sp)$ is initial among all continuous six-functor formalisms on the category of locally compact Hausdorff spaces. Here, continuous six-functor formalisms are those valued in dualizable presentable stable $\infty$-categories and satisfying canonical descent, profinite descent, and hyperdescent. As an application, we generalize Efimov’s computation of the algebraic $K$-theory of sheaves to all localizing invariants on continuous six-functor formalisms. Our results show that localizing invariants behave analogously to compactly supported sheaf cohomology theories when evaluated on continuous six-functor formalisms on locally compact Hausdorff spaces.
\end{abstract} 

\begingroup
\small
\tableofcontents
\endgroup

\section{Introduction}

\subsection*{Six‐functor formalism}

In its classical form, a six‐functor formalism organizes the basic operations in (co)homology theories. For example, if $X$ is a locally compact Hausdorff space, then the stable $\infty$‐category $\Shv(X; D(\Z))$ of sheaves of chain complexes on $X$ is symmetric monoidal (the tensor product is the sheafified pointwise tensor) and admits all (co)limits. In this setting, one has the six operations familiar from sheaf theory:

\[
  \otimes,\ \underline{\Hom} \qquad f^* \dashv f_* \qquad f_! \dashv f^!
\]

\begin{itemize}
  \item
    $\otimes$ and $\underline{\Hom}$: The tensor product $\otimes$ and the internal Hom make $\Shv(X; D(\Z))$ into a closed symmetric monoidal category. These are functors $\Shv(X; D(\Z))\times\Shv(X; D(\Z))\to\Shv(X; D(\Z))$ and $\Shv(X; D(\Z))^{\op}\times\Shv(X; D(\Z))\to\Shv(X; D(\Z))$, and serve as operations (1) and (2) of the six-functor formalism.
  \item
    $f_*$ and $f^*$ (direct/inverse image): A continuous map $f:Y\to X$ between locally compact Hausdorff spaces induces a pushforward $f_*: \Shv(Y; D(\Z))\to\Shv(X; D(\Z))$ (direct image) given by $\mathcal{F}\mapsto(\mathcal{F}\circ f^{-1})$, and a left adjoint $f^*: \Shv(X; D(\Z))\to\Shv(Y; D(\Z))$ (inverse image) given by sheafifying the preimage construction. Thus $f^*$ and $f_*$ are functors (3) and (4) of the formalism.
  \item
    $f_!$ and $f^!$ (proper pushforward/exceptional inverse): If $f$ is proper, then $f_*=f_!$. In general one defines a proper pushforward $f_!: \Shv(Y; D(\Z))\to\Shv(X; D(\Z))$ by ``sections with proper support'' (so that for $f:Y\to \pt$ this is $\Gamma_{\mathrm{c}}(Y,-)$) which is the functor (5) of the formalism. This admits a right adjoint $f^!:\Shv(X; D(\Z))\to\Shv(Y; D(\Z))$, the exceptional inverse image (functor (6)). In practice $f^!$ coincides with $f^*$ when $f$ is an open immersion, and $f_!=f_*$ when $f$ is proper.
\end{itemize}

These six operations satisfy the usual compatibilities, i.e., composition, base change and projection formula, forming the classical Grothendieck six-functor package on topological spaces.

Let $X$ be a locally compact Hausdorff space. Within this framework, the standard sheaf (co)homology theories of $X$ with $\Z$ coefficients can be expressed entirely in terms of the six functors as follows:
\[
  \hspace{-5em}
  \begin{aligned}
  \text{(sheaf cohomology)}\quad
  H^*(X,\mathbb{Z})
  &=p_*\,p^*\,\mathbb{Z},\\
  \text{(compactly supported sheaf cohomology)}\quad
  H^*_c(X,\mathbb{Z})
  &=p_!\,p^*\,\mathbb{Z},\\
  \text{(sheaf homology)}\quad
  H_*(X,\mathbb{Z})
  &=p_!\,p^!\,\mathbb{Z},\\
  \text{(locally finite sheaf/Borel--Moore homology)}\quad
  H^{\mathrm{lf}}_*(X,\mathbb{Z})
  &=p_*\,p^!\,\mathbb{Z}.
  \end{aligned}
\]
where $p : X \to \pt$ denotes the unique map to a point, and $\Z \in D(\Z)$ is the constant sheaf on the point.

\subsection*{Modern $\infty$‑categorical formalism}

Scholze \cite{scholze2022sixfunctors}, Liu-Zheng \cite{liu2015gluingrestrictednervesinftycategories}\cite{liu2024enhancedoperationsbasechange} and Mann \cite{mann2022padic} recently have recast the structure of six-functor formalism abstractly in higher categorical terms. One fixes a geometric context $(C,E)$, as in \cite[Appendix A.5]{mann2022padic}, where $C$ is an $\infty$‐category with finite limits (e.g. schemes, stacks, or spaces) and $E$ is a class of ``admissible'' morphisms (e.g. proper, étale maps or just all morphisms in $C$) stable under pullback. One then considers the symmetric monoidal $\infty$‐category $\Corr(C,E)$ of correspondences whose objects are those of $C$, and a morphism $X\to Y$ is given by a span $X\overset{f}{\leftarrow}W\overset{g}\rightarrow Y$ in $C$ with $g\in E$, composed by taking pullback. By definition, a 3‐functor formalism is a lax symmetric monoidal functor
\[
  D: \Corr(C, E)^\otimes \to \Cat_\infty
\]
for $(C,E)$ some geometric context. Unpacking this definition shows that for each $X\in C$, the $\infty$-category $D(X)$ is equipped with an external tensor product $\otimes$, and that for each $f:X\to Y$ in $C$ one has a functor $f^*:D(Y)\to D(X)$, while for each $f\in E$ one has a functor $f_!:D(X)\to D(Y)$, see \cite[Definition 2.4]{scholze2022sixfunctors}. In other words, a 3‐functor formalism produces the operations $\otimes$, $f^*$ and $f_!$ satisfying all evident compatibilities (base change and projection formula) by virtue of the functoriality of $D$.

One then says that $D$ is a six‐functor formalism if the remaining adjoints exist, i.e., the functors $-\otimes X$, $f^*$ and $f_!$ each admit right adjoints \cite[Definition 2.5]{scholze2022sixfunctors}. We summarize the definitions as follows.

\begin{definition}
  A 3‐functor formalism is a lax symmetric monoidal functor $D: \Corr(C, E)^\otimes \to \Cat_\infty$.
\end{definition}

\begin{definition}
  A 6‐functor formalism is a 3‐functor formalism for which $-\otimes X$, $f^*$ and $f_!$ all admit right adjoints.
\end{definition}

In this language, Liu–Zheng’s work \cite{liu2015gluingrestrictednervesinftycategories}\cite{liu2024enhancedoperationsbasechange}  and later Mann’s abstract theory \cite{mann2022padic} give concrete constructions of such functors $D$ from partial data. In practice, one often has two special classes of morphisms $I\subseteq E$ (e.g. open immersions) and $P \subseteq E$ (e.g. proper maps) with the property that any $f: X\to Y$ factors as $X\xrightarrow{i\in I}\bar X\xrightarrow{p\in P}Y$ satisfying extra assumptions\footnote{We omit the mild truncation assumption, which is automatically satisfied for all 1-categories, see the discussion below \cite[Example 4.1]{scholze2022sixfunctors} for the complete statement.}:
\begin{itemize}
  \item
    The classes of morphisms $I$ and $P$ are stable under pullback and composition and contain all equivalences.
  \item
    For any $f\in I$, the functor $f^*$ admits a left adjoint $f_!$ which satisfies base change and the projection formula.
  \item
    For any $f \in P$, the functor $f^*$ admits a right adjoint $f_*$ which satisfies base change and the projection formula.
  \item
    For any Cartesian diagram
    \[\begin{tikzcd}
      {X'} & X \\
      {Y'} & Y
      \arrow["{j'}", from=1-1, to=1-2]
      \arrow["{g'}"', from=1-1, to=2-1]
      \arrow["g", from=1-2, to=2-2]
      \arrow["j", from=2-1, to=2-2]
    \end{tikzcd}\]
    with $j \in I$ (hence $j'\in I$) and $g\in P$ (hence $g'\in P$), the natural map $j_!g'_* \to g_*j'_!$ is an isomorphism.
\end{itemize}

\begin{theorem}[{\cite[Theorem 4.6]{scholze2022sixfunctors}, \cite[Proposition A.5.10]{mann2022padic}, Liu–Zheng}]\label{thm:construction-six-functors}
  In that case, a given assignment $X\mapsto D(X)$ with monoidal structure extends canonically to a lax functor $D:\Corr(C,E)^\otimes\to\Cat_\infty$ by defining
  \[
    f_! = p_* i_!,
  \]
  where $i_!$ (resp.\ $p_*$) is left (resp.\ right) adjoint to $i^*$ (resp. $p^*$).
\end{theorem}

\subsection*{Universal property of sheaves}

Having exhibited all six standard functors on sheaves on locally compact Hausdorff spaces, a natural question is whether they fit together into an $\infty$-categorical six-functor formalism. To this end, we take all morphisms, the open immersions and the proper maps for the classes $E$, $I$ and $P$, respectively, and consider the assignment
\[
  \LCH^{\op} \to \CAlg(\Cat_\infty): X\mapsto \Shv(X; \Sp)
\]
which sends a locally compact Hausdorff space to its symmetric monoidal $\infty$-category of
$\Sp$-valued sheaves.

The fundamental structural ingredients required by the six-functor formalism, namely the symmetric monoidal structure, base change, projection formula, etc., were established by Volpe in \cite{volpe2023operationstopology}. This leads to the following formulation:

\begin{theorem}[{\cite[Theorem 7.4]{scholze2022sixfunctors}}]
  Under the setup described above, the conditions of \cref{thm:construction-six-functors} are satisfied, and one gets the resulting six-functor formalism\footnote{The statement of \cite[Theorem 7.4]{scholze2022sixfunctors} is formulated for sheaves valued in $D(\Z)$. However, as pointed out there, the entire proof works in a similar way with coefficients in the $\infty$-category of spectra or other coefficients. Alternatively, one can see \cite{volpe2023operationstopology} for a more detailed and general discussion.}
  \[
    X\mapsto \Shv(X; \Sp).
  \]
\end{theorem}

Drew and Gallauer \cite{Drew_2022} recently proved that Morel–Voevodsky’s stable $\A^1$-homotopy category $\mathrm{SH}$ realizes the universal coefficient system in motivic homotopy theory which gives rise to Grothendieck’s six operations. Specifically, $\mathrm{SH}$ is the initial object in $\CoSys^{\mathrm{c}}_B$ \cite[Theorem 7.14]{Drew_2022}. This naturally raises the question of whether an analogous universal six-functor formalism exists in the setting of locally compact Hausdorff spaces. Our main result provides an affirmative answer to this question.

We denote by $\SixFF(\LCH, E, I, P)$ the $\infty$-category of Nagata six-functor formalisms\footnote{Informally, Nagata six-functor formalisms $\SixFF(C, E, I, P)$ are those six-functor formalisms arising from the Liu–Zheng and Mann construction applied to the category $C$ with specified classes of morphisms $(E, I, P)$.} as defined in \cite[Definition 2.15]{dauser2025uniquenesssixfunctorformalisms}. When the context is clear, we omit the classes $E, I, P$ and simply write $\SixFF(\LCH)$. We then consider the subcategory $\SixFF(\LCH)^{\cont} \subseteq \SixFF(\LCH)$ consisting of all continuous six-functor formalisms, i.e., those valued in dualizable stable $\infty$-categories and satisfying canonical descent, profinite descent, and hyperdescent (\cref{def:continuous-six-functor}). The morphisms in this subcategory $\SixFF(\LCH)^{\cont}$ are natural transformations whose components preserve colimits. The functor $X\mapsto \Shv(X; \Sp)$ defines a continuous six-functor formalism. Our main result characterizes it as follows:

\begin{theorem}[{\cref{thm:sheaf-initial-six-functor}}]
  The object $\Shv(-; \Sp) \in \SixFF(\LCH)^{\cont}$ is initial.
\end{theorem}

Consequently, by the universal property of sheaves, both the cohomology and compactly supported cohomology theories in any continuous six-functor formalism coincide precisely with sheaf cohomology and sheaf cohomology with compact support, respectively (see \cref{prop:cohomology-equivalence}). Furthermore, the universal property of sheaves implies that all continuous six-functor formalisms are homotopy invariant (see \cref{thm:homotopy-invariant-six-functor}) which corresponds to the $\mathbb{A}^1$-homotopy invariance of the coefficient system in motivic homotopy theory.

As a concrete example, Efimov’s computation of the algebraic $K$-theory of $\infty$-categories of sheaves on locally compact Hausdorff spaces identifies it with sheaf cohomology with compact support \cite{efimov2025ktheorylocalizinginvariantslarge}. We generalize Efimov’s result to arbitrary continuous six-functor formalisms and to all continuous localizing invariants:

\begin{theorem}[{\cref{thm:localizing-invariant-six-functor}}]
  Fix a continuous six-functor formalism $D$. Let $\mathcal{E}$ be a presentable stable $\infty$-category, and $F^{\cont}: \Cat^{\dual}_\infty \to \mathcal{E}$ a continuous localizing invariant which commutes with filtered colimits. For any locally compact Hausdorff space $X$ there is an equivalence
  \[
    F^{\cont}(D(X)) \simeq \Gamma_{\mathrm{c}}\Big(X, \underline{F^{\cont}(D(\pt))}\Big).
  \]
\end{theorem}

Suppose $\mathcal{C}$ is a dualizable stable $\infty$-category. If we take $F^{\cont}$ to be the continuous algebraic $K$-theory $\mathrm{K}^{\cont}$ and set $D = \Shv(-; \mathcal{C})$, then one recovers Efimov’s result:

\begin{corollary}[{Efimov, \cite{efimov2025ktheorylocalizinginvariantslarge}}]
  For any locally compact Hausdorff space $X$ there is an equivalence
  \[
    \mathrm{K}^{\cont}(\Shv(X; \mathcal{C})) \simeq \Gamma_{\mathrm{c}}\Big(X, \underline{\mathrm{K}^{\cont}(\mathcal{C})}\Big).
  \]
\end{corollary}

In conclusion, the continuous six-functor formalism captures the essential categorical structure of sheaf theory.

\subsection*{Notation}

The framework of this paper is grounded in higher category theory. Throughout, we work in the language of $\infty$-categories à la Lurie (Higher Topos Theory \cite{lurie2009HTT}, Higher Algebra \cite{lurie2017HA}). We summarize our main notions and notation:

\begin{itemize}
  \item
    A stable $\infty$-category is a pointed $\infty$-category admitting all finite limits and colimits, and every pullback square is also a pushout (an exact square). Functors preserving these limits and colimits are exact.

  \item
    Write $\Cat_\infty$ for the $\infty$-category of (small) $\infty$-categories, and $\Cat^{\ex}_\infty$ for its (non-full) subcategory of stable $\infty$-categories with exact functors. Let $\Cat^{\perf}_\infty$ be the full subcategory of $\Cat^{\ex}_\infty$ of idempotent-complete stable $\infty$-categories. For an $\infty$-category $\mathcal{C}$, we denote $\Ind(\mathcal{C})$ its $\Ind$-construction and $\mathcal{C}^\kappa$ its full subcategory of $\kappa$-compact object. An $\infty$-category $\mathcal{C}$ is compactly generated if the colimit functor $\Ind(\mathcal{C}^\omega) \to \mathcal{C}$ is an equivalence.

  \item
    Write $\PrL$ (resp. $\PrR$) for the $\infty$-category of presentable $\infty$-categories and left adjoints (resp. right adjoints), and $\PrL_{\st}$ (resp. $\PrR_{\st}$) for its (non-full) subcategory of presentable stable $\infty$-categories.

  \item
    A presentable stable $\infty$-category $\mathcal{C}$ is dualizable if any of the following equivalent holds\footnote{The second and third conditions are also known as the definition of a compactly assembled $\infty$-category.} (Lurie):
    \begin{enumerate}
      \item
        $\mathcal{C}$ is dualizable in $\PrL_{\st}$.

      \item
        $\mathcal{C}$ is a retract in $\PrL_{\st}$ of some compactly generated category.

      \item
        The colimit functor $\colim: \Ind(\mathcal{C}) \to \mathcal{C}$ admits a left adjoint $\hat{y}: \mathcal{C} \to \Ind(\mathcal{C})$.
    \end{enumerate}
    Write $\Cat^{\dual, \mathrm{L}}_\infty$ for the full subcategory of $\PrL_{\st}$ of dualizable presentable stable $\infty$-categories. Let $\Cat^{\dual}_\infty$ be the (non-full) subcategory of $\Cat^{\dual, \mathrm{L}}_\infty$ with strongly continuous functors, i.e., left adjoints whose right adjoints admit further right adjoints.

    Throughout this paper, whenever we refer to a stable $\infty$-category as dualizable, we always mean that it is dualizable in the sense of presentable $\infty$-categories.

  \item
    Finally, denote by $\mathcal{S}$ the $\infty$-category of spaces and by $\Sp$, its stabilization, the $\infty$-category of spectra.
\end{itemize}

\subsection*{Acknowledgments}

I would like to begin by expressing my sincere gratitude to Gijs Heuts for his inspiring and insightful discussions, which significantly shaped the direction and development of this work. I am also deeply grateful to Bram Mesland for his continuous support and encouragement throughout the course of this project. Finally, I wish to thank Josefien Kuijper and Liam Keenan for their interest in my research and the invaluable discussions we have shared.
\section{Category of sheaves} \label{sec:category-of-sheaves}

In this section, we briefly recall the fundamental properties of the $\infty$-category of sheaves used throughout the paper, and we give explicit constructions of the direct and inverse image functors as well as the proper pushforward and exceptional inverse functors.

\begin{definition}
  Let $\LCH$ denote the category of locally compact Hausdorff topological spaces with continuous maps.
\end{definition}

\begin{definition}\label{def:sheaf}
  Let $X$ be a locally compact Hausdorff space. The $\infty$-category of \emph{sheaves} with values in an $\infty$-category $\mathcal{C}$ is denoted by $\Shv(X; \mathcal{C})$. A sheaf $\mathcal{F} \in \Shv(X; \mathcal{C})$ is given by a functor
  \[
    \mathcal{F} : \Open(X)^{\op} \to \mathcal{C}
  \]
  satisfying the following descent conditions:

  \begin{enumerate}
    \item
      $\mathcal{F}(\emptyset) = 0$,
    \item
      for any open subsets $U, V \subseteq X$, the square
      \[\begin{tikzcd}
        {\mathcal{F}(U \cup V)} & {\mathcal{F}(U)} \\
        {\mathcal{F}(V)} & {\mathcal{F}(U \cap V)}
        \arrow[from=1-1, to=1-2]
        \arrow[from=1-1, to=2-1]
        \arrow["\lrcorner"{anchor=center, pos=0.125}, draw=none, from=1-1, to=2-2]
        \arrow[from=1-2, to=2-2]
        \arrow[from=2-1, to=2-2]
      \end{tikzcd}\]
      is a pullback,
      \item for a filtered system of open subsets $U = \bigcup_{i \in I} U_i$,
      \[
        \mathcal{F}\Big( \bigcup_{i \in I} U_i \Big) \xrightarrow{\sim} \lim_{i \in I} \mathcal{F}(U_i)
      \]
      is an equivalence in $\mathcal{C}$.
  \end{enumerate}
  In the special case where $\mathcal{C}$ is the $\infty$-category of spaces $\Space$, we simply write $\Shv(X) := \Shv(X; \mathcal{S})$.
\end{definition}

\begin{definition}
  Dually, the $\infty$-category of \emph{cosheaves} on a locally compact Hausdorff space $X$ with values in an $\infty$-category $\mathcal{C}$ is the $\infty$-category of colimit-preserving functors out of sheaves
  \[
    \CoShv(X; \mathcal{C}) := \Fun^{\colim}(\Shv(X, \Space), \mathcal{C}),
  \]
  similarly we write $\CoShv(X) := \CoShv(X; \mathcal{S})$.
\end{definition}

Our main interest lies in the algebraic behavior of sheaves, meaning we primarily consider  sheaves valued in a (dualizable) stable $\infty$-category, e.g., spectra $\Sp$. A key feature is that sheaves with values in a presentable $\infty$-category $\mathcal{C}$ can be formed by tensoring in presentable $\infty$-categories. Concretely, one has
\[
  \Shv(X; \mathcal{C}) \simeq \Shv(X) \otimes \mathcal{C}
\]
as presentable $\infty$-categories \cite[Corollary 2.24]{volpe2023operationstopology}. When we have a dualizable stable $\infty$-category $\mathcal{C}$ as coefficient, the $\infty$-category of sheaves $\Shv(X; \mathcal{C})$ is in fact dualizable \cite[Proposition 2.2.20, Theorem 2.9.2]{krause2024sheaves}.

\begin{remark}
  A sheaf can equivalently be defined by the descent property associated to the open-cover Grothendieck topology. More concretely, one requires that for any covering sieve $R\to y(x)$ the canonical morphism
  \[
    \mathcal{F}(x) \to \lim_{y(x') \to R} \mathcal{F}(x')
  \]
  is an equivalence. In addition, one obtains $\CoShv(X; \mathcal{C}) \simeq \Shv(X; \mathcal{C}^{\op})^{\op}$ (see \cite[Remark 2.19]{volpe2023operationstopology}).
\end{remark}

\begin{construction}\label{cons:sheaf-of-sheaves}
  Let $i: X \to Y$ be an open immersion, and let $\mathcal{C}$ be a dualizable stable $\infty$-category. The induced inclusion $i: \Open(X) \subseteq \Open(Y)$ admits a right adjoint given by the inverse image $i^{-1}: \Open(Y) \to \Open(X)$. This adjunction induces an adjoint triple on presheaves
  \[\begin{tikzcd}[column sep=huge]
    {\PSh(X; \mathcal{C})} & {\PSh(Y; \mathcal{C})}
    \arrow[""{name=0, anchor=center, inner sep=0}, "{\Lan_{i^{\op}}}"{description}, bend left=30, shift left=3, from=1-1, to=1-2]
    \arrow[""{name=1, anchor=center, inner sep=0}, "{(i^{-1,\op})^*}"{description}, bend right=30, shift right=3, from=1-1, to=1-2]
    \arrow[""{name=2, anchor=center, inner sep=0}, "{(i^{\op})^*}"{description}, from=1-2, to=1-1]
    \arrow["\bot"{description}, draw=none, from=1, to=2]
    \arrow["\bot"{description}, draw=none, from=2, to=0]
  \end{tikzcd}\]
  where the left Kan extension of $i^{\op}$ agrees with the extension by zero. Restrict the functors to sheaves and pass to the sheafification, it gives rise to the adjoint triple on sheaves.
  \[\begin{tikzcd}[column sep=huge]
    {\Shv(X; \mathcal{C})} & {\Shv(Y; \mathcal{C})}
    \arrow[""{name=0, anchor=center, inner sep=0}, "{i_!}"{description}, bend left=30, shift left=3, from=1-1, to=1-2]
    \arrow[""{name=1, anchor=center, inner sep=0}, "{i^*}"{description}, bend right=30, shift right=3, from=1-1, to=1-2]
    \arrow[""{name=2, anchor=center, inner sep=0}, "{i_*}"{description}, from=1-2, to=1-1]
    \arrow["\bot"{description}, draw=none, from=1, to=2]
    \arrow["\bot"{description}, draw=none, from=2, to=0]
  \end{tikzcd}\]
  For an arbitrary continuous map $f: X \to Y$, such an adjoint triple does not exist in general. Nevertheless, one can still define the pushforward functor
  \[
    f_\ast: \Shv(X; \mathcal{C}) \to \Shv(Y; \mathcal{C}) \qquad (f_\ast \mathcal{F})(U) = \mathcal{F}(f^{-1}(U)),
  \]
  which admits a left adjoint, namely the pullback functor
  \[
    f^\ast: \Shv(Y; \mathcal{C}) \to \Shv(X; \mathcal{C}) \qquad (f^\ast \mathcal{F}) = \Big(V \mapsto \colim_{\substack{f(V) \subseteq U, \\ U\;\open}}\mathcal{F}(U) \Big)^{\wedge}.
  \]
  where $(-)^{\wedge}$ denotes sheafification. Then we obtain a functor
  \begin{equation}\label{cons:functor-of-sheaves}
    \LCH^{\op} \to \PrL, \qquad X \mapsto \Shv(X; \mathcal{C})
  \end{equation}
  endowed with the contravariant pullback functoriality.
\end{construction}

\begin{proposition}\label{prop:sheaf-of-sheaves}
  The assignment (\ref{cons:functor-of-sheaves}) defines a sheaf of (presentable) $\infty$-categories on $\LCH$, i.e.,
  \begin{enumerate}
    \item
      $\Shv(\emptyset; \mathcal{C}) = \pt$,
    \item
      for any open subsets $U, V \subseteq X$, the square
      \[\begin{tikzcd}
        {\Shv(U \cup V;\mathcal{C})} & {\Shv(U;\mathcal{C})} \\
        {\Shv(V;\mathcal{C})} & {\Shv(U \cap V;\mathcal{C})}
        \arrow[from=1-1, to=1-2]
        \arrow[from=1-1, to=2-1]
        \arrow["\lrcorner"{anchor=center, pos=0.125}, draw=none, from=1-1, to=2-2]
        \arrow[from=1-2, to=2-2]
        \arrow[from=2-1, to=2-2]
      \end{tikzcd}\]
    is a pullback in $\Cat_\infty^{\ex}$. Moreover, all the functors are Bousfield localizations.
    \item
      for a filtered system of open subsets $U = \bigcup_{i \in I} U_i$,
      \[
        \Shv(U; \mathcal{C}) \xrightarrow{\sim} \lim_{i \in I} \Shv(U_i, \mathcal{C})
      \]
      is an equivalence, where the limit is taken along the functors $(f_{ij})^\ast: \Shv(U_j; \mathcal{C}) \to \Shv(U_i; \mathcal{C})$ in $\Cat_\infty^{\ex}$.
  \end{enumerate}
\end{proposition}

\begin{proof}
  \cite[Proposition 3.6.3]{krause2024sheaves}
\end{proof}

For an open immersion $i: X\to Y$, we have an adjoint triple $i_! \dashv i^* \dashv i_*$ as established in \cref{cons:sheaf-of-sheaves}. Thus the functor $i_!$ is compactly assembled\footnote{In stable case being compactly assembled is equivalent to being strongly continuous, i.e., the right adjoint has a further right adjoint, see the discussion below \cite[Definition 2.9.6]{krause2024sheaves}.}, i.e., takes compact morphisms to compact morphisms.

Consider the category $\LCH_{\open}$ of locally compact Hausdorff spaces and open immersions, the assignment $X \mapsto \Shv(X; \mathcal{C})$ on $\LCH_{\open}$ with values landing in $\Cat^{\dual, \mathrm{L}}_\infty$ becomes a covariant functor by means of the extension functoriality.

\begin{corollary}
  The functor $\Shv(-; \mathcal{C}): \LCH_{\open} \to \Cat^{\dual, \mathrm{L}}_\infty$ is a cosheaf.
\end{corollary}

\begin{proof}
  It follows directly from \cref{prop:sheaf-of-sheaves} by passing to the left adjoints.
\end{proof}

\begin{remark}\label{rem:recollement}
  For any open immersion $j: U \subseteq X$ with complementary closed immersion $i: (X \setminus U) \to X$, there is a fiber sequence in $\Cat^\ex_\infty$
  \[
    \Shv(X \setminus U; \mathcal{C}) \xrightarrow{i_*} \Shv(X; \mathcal{C}) \xrightarrow{j^*} \Shv(U; \mathcal{C})
  \]
  which follows directly from the recollement, i.e., the exact sequence $i_* i^! \to \id_{\Shv(X;\mathcal{C})} \to j_*j^*$ together with the fact that $i_*$ is fully faithful (see \cite[Corollary 4.7, Remark 4.11]{volpe2023operationstopology}). This fiber sequence is also called the localization sequence. Then the second sheaf condition in the statement of \cref{prop:sheaf-of-sheaves} can be deduced from the localization sequence by comparing the fibers.
\end{remark}

\begin{proposition}[Profinite descent]\label{prop:sheaf-profinite-descent}
  The sheaf $\Shv(-; \mathcal{C})$ satisfies profinite descent, specifically, for a cofiltered limit of compact Hausdorff spaces $X = \lim_{i \in I} X_i$ we have an equivalence
  \[
    \Shv(X; \mathcal{C}) \xrightarrow{\simeq} \lim\nolimits^{\PrR} \Shv(X_i; \mathcal{C})
  \]
  where the limit is taken along the functors $(f_{ij})_\ast: \Shv(X_i; \mathcal{C}) \to \Shv(X_j; \mathcal{C})$.
\end{proposition}

\begin{proof}
  \cite[Proposition 3.6.7]{krause2024sheaves}
\end{proof}

\begin{proposition}\label{prop:section-profinite-descent}
  For a cofiltered limit of compact Hausdorff spaces $X = \lim_{i \in I} X_i$ we have an equivalence
  \[
    \Gamma(X, \underline{E}) \simeq \colim \Gamma(X_i, \underline{E})
  \]
  for any spectrum $E$.
\end{proposition}

\begin{proof}
  \cite[Proposition 3.6.8]{krause2024sheaves}
\end{proof}

In the theory of $\infty$-categorical sheaves, equivalences cannot, in general, be detected simply by evaluating on stalks. One counterexample can be found in \cite[Counterexample 6.5.4.5, Remark 6.5.4.7]{lurie2009HTT}\footnote{The topological space in this counterexample is coherent but not Hausdorff.}. To enforce a fully local criterion, one further localizes the $\infty$-category of sheaves with respect to hypercovers. The resulting objects are the so-called hypersheaves, which by construction satisfy the hyperdescent (see \cite[\Sec6.5.4]{lurie2009HTT} for more details).

\begin{definition}\label{def:hypercomplete-space}
  A locally compact Hausdorff space $X$ is \emph{hypercomplete} if the sheaf $\Shv(X)$ is hypercomplete, equivalently, if the canonical functor $\Shv(X) \to \Shv^{\hyp}(X)$ is an equivalence.
\end{definition}

\begin{lemma}
  If $X$ is paracompact and has finite covering dimension, then it is hypercomplete.
\end{lemma}

\begin{proof}
  \cite[Theorem 7.2.3.6, Proposition 7.2.1.10, Corollary 7.2.1.12]{lurie2009HTT}
\end{proof}

\begin{lemma}\label{lem:stalkwise-check}
  Equivalences between hypersheaves valued in a compactly assembled category $\mathcal{C}$ can be detected at the level of stalks. Consequently, if a space $X$ is hypercomplete, then equivalences in $\Shv(X, \mathcal{C})$ can be determined stalkwise.
\end{lemma}

\begin{proof}
  \cite[Lemma 3.6.11]{krause2024sheaves}
\end{proof}

\begin{remark}
  In the situation of \cref{def:hypercomplete-space}, i.e., $X$ is hypercomplete, one has
  \[
    D(\Ab(X)) \simeq \Shv(X; D(\Z)),
  \]
  the derived category of abelian sheaves on $X$ agrees with the $\infty$-category of sheaves on $X$ valued in the derived category $D(\Z)$ \cite[Proposition 7.1]{scholze2022sixfunctors}. Thus the theory of $\infty$-categories of sheaves subsumes the classical sheaf cohomology theory over ``good'' spaces, e.g., compact manifolds.
\end{remark}

The method of reducing the verification of equivalences of sheaves to the stalk level allowing local-to-global arguments to be carried out, together with the profinite descent, will be the central techniques in this paper.

To conclude this section, we present Lurie’s formulation of Verdier duality, which establishes the self-duality of categories of sheaves.

\begin{construction}
  Let $X$ be a Hausdorff topological space and let $\mathcal{F}$ be a $\mathcal{C}$-valued sheaf on $X$. Given an open subset $U \subseteq X$, we let
  \[
    \mathcal{F}_{\mathrm{c}}(U) = \Gamma_{\mathrm{c}}(U, \mathcal{F}|_U)
  \]
  which are the sections of $\mathcal{F}$ supported in a compact subset of $U$. Then $\mathcal{F}_{\mathrm{c}}$ is a covariant functor in $U$.
\end{construction}

\begin{theorem}[Lurie, Verdier duality]\label{thm:verdier-duality}
  Let $\mathcal{C}$ be a stable $\infty$-category which admits small limits and colimits, and let $X$ be a locally compact Hausdorff space. The assignment $\mathcal{F}\mapsto \mathcal{F}_{\mathrm{c}}$ determines an equivalence of $\infty$-categories
  \[
    \D: \Shv(X; \mathcal{C}) \simeq \CoShv(X; \mathcal{C}).
  \]
\end{theorem}

\begin{proof}
  See \cite[Theorem 5.5.5.1]{lurie2017HA} or \cite[Theorem 5.10]{volpe2023operationstopology}.
\end{proof}

\begin{construction}
  Let $f: X \to Y$ be a continuous map. The adjunction $f_\ast \dashv f^\ast$ dualizes to an adjunction on the categories of cosheaves
  \[\begin{tikzcd}[sep=2.5em]
    {f_+:\CoShv(Y;\mathcal{C})} & {\CoShv(X;\mathcal{C}): f^+}
    \arrow[""{name=0, anchor=center, inner sep=0}, shift left=2, from=1-1, to=1-2]
    \arrow[""{name=1, anchor=center, inner sep=0}, shift left=2, from=1-2, to=1-1]
    \arrow["\bot"{description}, draw=none, from=1, to=0]
  \end{tikzcd}\]
  Concretely the functor $f_+$ is given by $f_+(\mathcal{F})(U) = \mathcal{F}(f^{-1}(U))$. By Verdier duality (\cref{thm:verdier-duality}), the adjunction $f_+ \dashv f^+$ on cosheaves corresponds to an adjunction $f_! \dashv f^!$ on sheaves. More abstractly, the adjunction $f_! \dashv f^!$ can be viewed as the dual of the adjunction $f_*\dashv f^*$ under the canonical self-duality of categories of sheaves.
\end{construction}

\begin{remark}
  When $f: X \to Y$ is a proper map, we have $f_! = f_*$, and thus obtain an adjoint triple $f^\ast \dashv f_\ast \dashv f^!$ (see \cite[Remark 6.6]{volpe2023operationstopology}). In this case, $f^*$ is strongly continuous.
\end{remark} 
\section{Coefficient system} \label{sec:coefficient-system}

To characterize the universal property of sheaves, we adapt the definition of coefficient systems from \cite[Definition 7.5]{Drew_2022}, originally formulated for finite type schemes over a noetherian scheme of finite Krull dimension, to the setting of locally compact Hausdorff spaces.

\begin{definition}[Coefficient system]\label{def:coeff-system}
  A functor $C: \LCH^{\op} \to \mathrm{CAlg}(\Cat_\infty^{\ex})$ taking values in symmetric monoidal stable $\infty$-categories and exact symmetric monoidal functors is called a \emph{coefficient system} if it satisfies the following properties:

  \begin{enumerate}
    \item
      \begin{enumerate}
        \item (\emph{Pushforwards})
          For every morphism $f:X \to Y$ in $\LCH$, the pullback functor $f^*: C(Y) \to C(X)$ admits a right adjoint $f_*: C(X)\to C(Y)$.
        \item (\emph{Internal homs})
          For every $X\in\LCH$, the symmetric monoidal structure on $C(X)$ is closed.
      \end{enumerate}

    \item
      Whenever $i:X\to Y$ is an \emph{open immersion} in $\LCH$, the functor $i^*:C(Y)\to C(X)$ admits a left adjoint $i_!$, and:
      \begin{enumerate}
        \item (\emph{Base change})
          For each Cartesian square
          \[\begin{tikzcd}
            {X'} & {Y'} \\
            X & Y
            \arrow["{i'}", from=1-1, to=1-2]
            \arrow["{f'}"', from=1-1, to=2-1]
            \arrow["\lrcorner"{anchor=center, pos=0.125}, draw=none, from=1-1, to=2-2]
            \arrow["f", from=1-2, to=2-2]
            \arrow["i", from=2-1, to=2-2]
          \end{tikzcd}\]
          the canonical exchange transformation $i'_!\,(f')^* \to f^*\,i_!$ is an equivalence.
        \item (\emph{Projection formula})
          The canonical map
          \[
            i_!(i^*(-)\otimes -) \to - \otimes i_!(-)
          \]
          is an equivalence of functors $C(Y)\times C(X)\to C(Y)$.
      \end{enumerate}

    \item (\emph{Localization})
      \begin{enumerate}
        \item
          $C(\emptyset) = 0$.
        \item
          For each closed immersion $i: Z\hookrightarrow X$ with open complement $j: U\hookrightarrow X$, and the square
          \[
            \begin{tikzcd}
              C(Z) \ar{r}{i_*} \ar{d} & C(X) \ar{d}{j^*}\\
              0 \ar{r}                & C(U)
            \end{tikzcd}
          \]
          is Cartesian in $\Cat_\infty^{\ex}$.
        \item
          For an increasing, filtered union $U = \bigcup_{i\in I} U_i$ of opens in $X$ we have that
          \[
            C(U) \xrightarrow{\simeq} \lim\nolimits^{\Cat_\infty^{\ex}} C(U_i)
          \]
          where the limit is taken along the functors $(f_{ij})^\ast: D(U_j) \to D(U_i)$.
      \end{enumerate}
  \end{enumerate}
  A morphism of coefficient systems is a natural transformation $\phi: C\to C'$ such that, for each open immersion $i: X'\to X$, the exchange transformation
  \[
    i_! \,\phi_{X'} \to \phi_X \,i_!
  \]
  is an equivalence. This defines a sub-\(\infty\)-category $\CoSys \subseteq \Fun(\LCH^{\op},\,\CAlg(\Cat^{\ex}_{\infty}))$. We define the cocomplete coefficient systems $\CoSys^{\mathrm{c}}$ exactly the same as \cite[Definition 7.7]{Drew_2022}.
\end{definition}

\begin{remark} \label{rem:verdier-sequence-from-localization}
  Let $j: U \hookrightarrow X$ be an open immersion. Then $j^*$ has a fully faithful left adjoint $j_!$. Indeed, consider the following Cartesian square
  \[\begin{tikzcd}
    U & U \\
    U & X
    \arrow["{\id_U}", from=1-1, to=1-2]
    \arrow["{\id_U}"', from=1-1, to=2-1]
    \arrow["\lrcorner"{anchor=center, pos=0.125}, draw=none, from=1-1, to=2-2]
    \arrow["j", from=1-2, to=2-2]
    \arrow["j", from=2-1, to=2-2]
  \end{tikzcd}\]
  By the base change, the unit
  \[
    \id_U \xrightarrow{\simeq} j^* j_!.
  \]
  is an equivalence. By \cite[Lemma A.2.8]{calmès2025hermitianktheorystableinftycategories}, the Cartesian square in the localization axiom is also coCartesian. In this context, the sequence
  \[
    C(Z) \xrightarrow{i_*} C(X) \xrightarrow{j^*} C(U)
  \]
  is called a Verdier sequence, also known as a localization sequence. Moreover, since $j^*$ admits both a left and a right adjoint, this sequence is a split Verdier sequence. For a modern and comprehensive reference on Verdier sequences, see Appendix A of \cite{calmès2025hermitianktheorystableinftycategories}. We will also recall the relevant definitions and further explore these notions in \cref{sec:localizing-invariant}.
\end{remark}

\begin{remark}
  The localization axiom of coefficient systems can be interpreted as the categorical analog of the sheaf axiom (cf. \cref{def:sheaf}) in the context of large $\infty$-categories. For this reason, we refer to it interchangeably as canonical descent.
\end{remark}

We employ the generic construction introduced in \cite{Drew_2022}, specifying to our context of locally compact Hausdorff spaces, to describe the initial object in the coefficient system. The strategy is to embed the $\infty$-category of coefficient systems into the larger $\infty$-category of pullback formalisms.

\begin{convention}\label{conv:coeff-system-context}
  We work with the finitely complete 1-category $\LCH$ and take $I$ to be the class of all open immersions together with a set $\mathcal{L}$ of descent diagrams associated to the open‐cover Grothendieck topology \cite[Example 5.19]{Drew_2022}.

  In the framework of \cite{Drew_2022}, we set $S = \LCH$, $P = I$, $\T = \{(S^1, 1)\}$ and $\mathcal{L}$ as described above (see \cite[Convention 7.1]{Drew_2022}).
\end{convention}

We now briefly review the pullback formalism as presented in \cite{Drew_2022}.

\begin{definition}[{Pullback formalism, \cite[Definition 2.11]{Drew_2022}}]
  A \emph{pullback formalism} is a functor $C: \LCH^{\op} \to \mathrm{CAlg}(\Cat_\infty)$ such that for every open immersion $i:X\to Y$ in $\LCH$, the functor $i^*:C(Y)\to C(X)$ admits a left adjoint $i_!$, and
  \begin{enumerate}
    \item (\emph{Base change})
      For each Cartesian square
      \[\begin{tikzcd}
        {X'} & {Y'} \\
        X & Y
        \arrow["{i'}", from=1-1, to=1-2]
        \arrow["{f'}"', from=1-1, to=2-1]
        \arrow["\lrcorner"{anchor=center, pos=0.125}, draw=none, from=1-1, to=2-2]
        \arrow["f", from=1-2, to=2-2]
        \arrow["i", from=2-1, to=2-2]
      \end{tikzcd}\]
      the canonical exchange transformation $i'_!\,(f')^* \to f^*\,i_!$ is an equivalence.
    \item (\emph{Projection formula})
      The canonical map
      \[
        i_!(i^*(-)\otimes -) \to - \otimes i_!(-)
      \]
      is an equivalence of functors $C(Y)\times C(X)\to C(Y)$.
  \end{enumerate}
  A morphism of pullback formalisms is a natural transformation $\phi: C\to C'$ such that, for each open immersion $i: X'\to X$, the exchange transformation
  \[
    i_! \,\phi_{X'} \to \phi_X \,i_!
  \]
  is an equivalence. This defines a sub-\(\infty\)-category $\PB \subseteq \Fun(\LCH^{\op},\,\CAlg(\Cat_{\infty}))$.
\end{definition}

\begin{convention}
  By imposing different conditions on the category of pullback formalisms $\PB$, we obtain a list of variants:
  \begin{itemize}
    \item
      $\PB^{\mathrm{c}}$, cocomplete pullback formalisms \cite[Definition 4.1]{Drew_2022};
    \item
      $\PB^{\mathrm{c}}_{\mathcal{L}}$, pullback formalisms satisfying the $\mathcal{L}$-descent \cite[Definition 5.4]{Drew_2022};
    \item
      $\PB^{\mathrm{pt}}$, pointed pullback formalisms \cite[Definition 6.1]{Drew_2022};
    \item
      $\PB^{\mathrm{c},\mathrm{pt}}_{\T}$, $\T$-stable pullback formalisms \cite[Definition 6.10]{Drew_2022}.
  \end{itemize}
\end{convention}

\begin{theorem}\label{thm:initial-pullback-formalism}
  We have a sequence of adjunctions between (very large) presentable sub-$\infty$-categories of $\Fun(\LCH^{\op},\CAlg(\Cat_{\infty}))$, their initial objects, and the respective evaluations of these initial objects at a space $X$ as shown below:
  \[\begin{tikzcd}[row sep = small]
    \PB & {\PB^{\mathrm{c}}} & {\PB^{\mathrm{c}}_{\mathcal{L}}} & {\PB^{\mathrm{c},\mathrm{pt}}_{\mathcal{L}}} & {\PB^{\mathrm{c},\mathrm{pt}}_{\mathcal{L},S^1}} \\
    {C_{\mathrm{gm}}} & {\widehat{C}_{\mathrm{gm}}} & {L_{\mathcal{L}}\widehat{C}_{\mathrm{gm}}} & {L_{\mathrm{pt}} L_{\mathcal{L}}\widehat{C}_{\mathrm{gm}}} & {L_{S^1} L_{\mathrm{pt}} L_{\mathcal{L}}\widehat{C}_{\mathrm{gm}}} \\
    {\Open(X)} & {\PSh(X)} & {\Shv(X)} & {\Shv(X)_\ast} & {\Shv(X; \Sp)}
    \arrow[shift left=3, bend left=30, from=1-1, to=1-2]
    \arrow[hook', from=1-2, to=1-1]
    \arrow[shift left=3, bend left=30, from=1-2, to=1-3]
    \arrow[hook', from=1-3, to=1-2]
    \arrow[shift left=3, bend left=30, from=1-3, to=1-4]
    \arrow[hook', from=1-4, to=1-3]
    \arrow[shift left=3, bend left=30, from=1-4, to=1-5]
    \arrow[hook', from=1-5, to=1-4]
    \arrow[maps to, from=2-1, to=2-2]
    \arrow[maps to, from=2-2, to=2-3]
    \arrow[maps to, from=2-3, to=2-4]
    \arrow[maps to, from=2-4, to=2-5]
  \end{tikzcd}\]
  Moreover, the last three right adjoints are fully faithful, thus describing reflexive sub-$\infty$-categories.
\end{theorem}

\begin{proof}
  It is a special case of \cite[Theorem 7.3]{Drew_2022} specifying to \cref{conv:coeff-system-context}, see \cite[Construction 7.2]{Drew_2022} for the detailed construction for each step.
\end{proof}

\begin{proposition}\label{prop:sheaf-is-cosys}
  The assignment $X \mapsto \Shv(X; \Sp)$, endowed with the contravariant pullback functoriality, defines a coefficient system.
\end{proposition}

\begin{proof}
  The base change and projection formula hold by construction since $\Shv(-; \Sp)$ is a pullback formalism. The localization follows directly from \cref{prop:sheaf-of-sheaves,rem:recollement}.
\end{proof}

\begin{theorem}\label{thm:sheaf-initial-cosys}
  The object $\Shv(-; \Sp) \in \CoSys^{\mathrm{c}}$ is initial.
\end{theorem}

\begin{proof}
  We adapt the proof of \cite[Proposition 7.13]{Drew_2022}. Since the pullback formalism $\Shv(-; \Sp)$ is an object of $\CoSys^{\mathrm{c}}$, by \cref{thm:initial-pullback-formalism}, it suffices to show that $\CoSys^{\mathrm{c}}$ is a full subcategory of $\PB^{\mathrm{c},\mathrm{pt}}_{\mathcal{L},S^1}$. The only non-trivial statement that remains to prove is that if $C$ is a cocomplete coefficient system, then it satisfies open-cover descent.

  Let $u: \Delta^{+,\op} \to \LCH$ be a \v{C}ech semi-nerve associated to an open cover of $U$ and $M \in C(U)$. It suffices to prove that the morphism
  \begin{equation}\label{equ:descent-colimit}
    \colim_{i\in \Delta^{+,\op}} u(i)_! u(i)^* M \to M
  \end{equation}
  is an equivalence by \cite[Remark 5.16]{Drew_2022}. By the Yoneda lemma, it is equivalent to show the functor $\Map_{C(U)}(-, N)$ takes (\ref{equ:descent-colimit}) to an equivalence for an arbitrary $N \in C(U)$. Consider the composite
  \[
    F_{U,M,N}\colon
    \LCH^{\op}
    \;\xrightarrow{[-]}\;
    C(U)^{\op}
    \;\xrightarrow{-\otimes M}\;
    C(U)^{\op}
    \;\xrightarrow{\Map_{\mathcal C(U)}(-,N)}\;
    \Space
  \]
  where the first functor $[-]$ is induced by the essentially unique morphism of pullback formalisms, and sends $i: V \to U$ to $i_!i^* 1_{C(U)}$ (see \cite[Theorem 3.26]{Drew_2022}). It follows that $C$ satisfies open-cover descent if and only if $F_{U,M,N}$ has the open-cover descent property for all $U,M,N$. Note that the open-cover descent property of $F_{U,M,N}$ means precisely that $F_{U,M,N}$ is a sheaf on $U$. Translating back, it suffices to show that $C$ satisfies a `sheaf'-property, and this follows easily from the localization property.
\end{proof}

\begin{remark}
  For any $C\in\CoSys^{\mathrm{c}}$, the unique morphism $[-]: \Shv(-; \Sp) \to C$ can be described explicitly. Let $X$ be a locally compact Hausdorff space and $i: U \to X$ be an open immersion. Then
  \[
    \big[\Sigma^\infty_+ \Map(-, U)\big] = i_! i^* 1_{C(X)}.
  \]
  Moreover, since $\Shv(X; \Sp)$ is generated by representable spectral sheaves under filtered colimits, the above formula uniquely determines the morphism $[-]$.
\end{remark}

\begin{proposition}\label{prop:free-forgetful-cosys-adjunction}
  Let $U: \CoSys^{\mathrm{c}} \to \CAlg(\Cat^{\ex, \mathrm{c}}_{\infty})$ denote the forgetful functor given by evaluation at the point. Then $U$ is right adjoint to the functor
  \[
    F: \CAlg(\Cat^{\ex, \mathrm{c}}_{\infty}) \to \CoSys^{\mathrm{c}},\quad \mathcal{C} \mapsto \Shv(X; \Sp) \otimes \mathcal{C}.
  \]
  In other words, the assignment $X \mapsto \Shv(X; \Sp) \otimes \mathcal{C}$ exhibits the free coefficient system generated by the cocomplete symmetric monoidal stable $\infty$-category $\mathcal{C}$. In particular, when $\mathcal{C}$ is presentable, the assignment $X \mapsto \Shv(-; \mathcal{C})$ is the free coefficient system generated by $\mathcal{C}$.
\end{proposition}

\begin{proof}
  The unit is the identity. For the counit, let $C\in\CoSys^{\mathrm{c}}$ and denote $\mathcal{C} = C(\pt)$. We need to construct an natural transformation
  \[
    \epsilon: \Shv(-; \Sp) \otimes \mathcal{C} \to C.
  \]
  Note that for every locally compact Hausdorff space $X$, we can view $C(X)$ as an algebra over $\mathcal{C}$ via $p^*: C(\pt) \to C(X)$ where $p: X \to \pt$ is the unique map. We then define the desired morphism
  \[
    \epsilon_X: \Shv(X; \Sp) \otimes \mathcal{C} \to C(X)
  \]
  as the composition of the tensor product of the unique map $\Shv(X; \Sp) \to C(X)$ with the algebra map $p^*: C(\pt) \to C(X)$, followed by the multiplication map in $C(X)$. In addition, we need to verify the morphism $\epsilon$ is compatible with $i_!$ for every open immersion $i: U \to X$, i.e., the exchange transformation $i_! \,\epsilon_U \to \epsilon_X\,i_!$ is an equivalence. It is equivalent to show the following outer square commutes.
  \[\begin{tikzcd}[sep=large]
    {\Shv(U;\Sp)\otimes \mathcal{C}} & {C(U)\otimes \mathcal{C}} & {C(U)\otimes C(U)} & {C(U)} \\
    {\Shv(X;\Sp)\otimes \mathcal{C}} & {C(X)\otimes \mathcal{C}} & {C(X)\otimes C(X)} & {C(X)}
    \arrow["{[-]\otimes\id_{\mathcal{C}}}", from=1-1, to=1-2]
    \arrow["{i_!\otimes\id_{\mathcal{C}}}"', from=1-1, to=2-1]
    \arrow["{\id_{C(U)} \otimes p^*}", from=1-2, to=1-3]
    \arrow["{i_!\otimes\id_{\mathcal{C}}}"', from=1-2, to=2-2]
    \arrow[from=1-3, to=1-4]
    \arrow["{i_!}", from=1-4, to=2-4]
    \arrow["{[-]\otimes\id_{\mathcal{C}}}", from=2-1, to=2-2]
    \arrow["{\id_{C(U)} \otimes p^*}", from=2-2, to=2-3]
    \arrow[from=2-3, to=2-4]
  \end{tikzcd}\]
  The above left square clearly commutes, so it suffices to show that the above right square commutes. Now consider the following diagram
  \begingroup
    \setlength\abovedisplayskip{2em}
    \setlength\belowdisplayskip{2em}
    \[
    \begin{tikzcd}[nodes=overlay, row sep=3.5em, column sep=6em]
      & {C(U \times \pt)} && {C(U \times U)} && {C(U)} \\
      {C(U)\otimes \mathcal{C}} && {C(U)\otimes C(U)} && {C(U)} \\
      & {C(X \times \pt)} && {C(X \times X)} && {C(X)} \\
      {C(X)\otimes \mathcal{C}} && {C(X)\otimes C(X)} && {C(X)}
      \arrow["{\pi_1^*}", from=1-2, to=1-4]
      \arrow["{i_!}"{pos=0.7}, from=1-2, to=3-2]
      \arrow["{\Delta^*}"{pos=0.66}, from=1-4, to=1-6]
      \arrow["{\Delta^*}"{pos=0.66}, from=3-4, to=3-6]
      \arrow["{i_!}"{pos=0.7}, from=1-6, to=3-6]
      \arrow[from=2-1, to=1-2]
      \arrow["{\id_{C(U)} \otimes p^*}", crossing over, from=2-1, to=2-3]
      \arrow["{i_!\otimes\id_{\mathcal{C}}}"', from=2-1, to=4-1]
      \arrow[from=2-3, to=1-4]
      \arrow[from=2-3, to=2-5]
      \arrow["{=}"{description}, from=2-5, to=1-6]
      \arrow["{i_!}"{pos=0.7}, crossing over, from=2-5, to=4-5]
      \arrow["{\pi_1^*}", from=3-2, to=3-4]
      \arrow[from=4-1, to=3-2]
      \arrow["{\id_{C(U)} \otimes p^*}", from=4-1, to=4-3]
      \arrow[from=4-3, to=3-4]
      \arrow[from=4-3, to=4-5]
      \arrow["{=}"{description}, from=4-5, to=3-6]
    \end{tikzcd}
    \]
  \endgroup
  then the commutativity of the front face of the cube follows from the commutativity of all five other faces. The triangle identities are straightforward to verify.
\end{proof}

\begin{corollary}\label{cor:factor-initial-cosys-map}
  Let $C\in\CoSys^{\mathrm{c}}$ and denote $\mathcal{C} = C(\pt)$. The unique morphism $[-]: \Shv(-; \Sp) \to C$ factors through the counit $\epsilon: \Shv(-; \Sp) \otimes \mathcal{C} \to C$.
  \[\begin{tikzcd}[row sep=large]
    {\Shv(-; \Sp) } & C \\
    {\Shv(-; \Sp) \otimes \mathcal{C} }
    \arrow["{[-]}"{description}, from=1-1, to=1-2]
    \arrow["{[-]}"{description}, from=1-1, to=2-1]
    \arrow["\epsilon"', from=2-1, to=1-2]
  \end{tikzcd}\]
\end{corollary}

\begin{corollary}\label{cor:free-cosys-over-base}
  The object $\Shv(-; \Sp) \otimes \mathcal{B}$ is initial in the category $\CoSys^{\mathrm{c}}_{\mathcal{B}}$, where $\CoSys^{\mathrm{c}}_{\mathcal{B}}$ denotes the full subcategory of $\CoSys^{\mathrm{c}}$ consisting of cocomplete coefficient systems valued in algebras over the base $\mathcal{B} \in \CAlg(\Cat^{\ex, \mathrm{c}}_{\infty})$.
\end{corollary} 
\section{Universal six-functor formalism} \label{sec:universal-six-functor-formalism}

As established in \cref{thm:sheaf-initial-cosys}, the sheaf is initial in the category of cocomplete coefficient systems. In this section, we present our main theorem, which extends this universal property to the $\infty$-category of continuous six-functor formalisms. Before delving into the proof, we recall a more formal description of the data that constitute a six-functor formalism. Specifically, we adopt the framework of Nagata six-functor formalisms developed in \cite{dauser2025uniquenesssixfunctorformalisms}.

\begin{definition}[{Nagata set‐up, \cite[Definition 2.1]{dauser2025uniquenesssixfunctorformalisms}}]
  A \emph{Nagata set‐up} $(\mathcal{C},E,I,P)$ is a geometric set‐up $(\mathcal{C},E)$ as in \cite[Appendix A.5]{mann2022padic}, together with subsets $I,P\subseteq E$ satisfying\footnote{We omit the mild truncation assumption, which automatically holds in the 1-categorical context, e.g., the category of locally compact Hausdorff spaces. See \cite[Definition 2.1]{dauser2025uniquenesssixfunctorformalisms} for the complete statement.}
  \begin{enumerate}
  \item
    any morphism $f\in E$ factors as $\bar f\circ j$ with $j\in I$ and $\bar f\in P$;
  \item
    $(C,I)$ and $(C,P)$ are geometric set‐ups;
  \item
    for morphisms $f:X\to Y$ in $\mathcal{C}$ and $g:Y\to Z$ in $I$, we have $f\in I$ if and only if $ g\circ f\in I$;
  \item
    for morphisms $f:X\to Y$ in $\mathcal{C}$ and $g:Y\to Z$ in $P$, we have $f\in P$ if and only if $ g\circ f\in P$;
  \end{enumerate}
\end{definition}

\begin{example}
  The tuple
  \[
    \big(\LCH,\;\{\text{all morphisms}\},\;\{\text{open immersions}\},\;\{\text{proper maps}\}\big)
  \]
  constitutes a Nagata set-up.
\end{example}

Consider the following commutative square (up to a specified isomorphism $j\circ f\cong h\circ g$). Suppose that
\begin{equation}\label{sq:adjoinable}
  \begin{tikzcd}
    A & B \\
    C & D
    \arrow["f", from=1-1, to=1-2]
    \arrow["g"', from=1-1, to=2-1]
    \arrow["j"', from=1-2, to=2-2]
    \arrow["h", from=2-1, to=2-2]
  \end{tikzcd}
\end{equation}
\begin{itemize}
  \item
    $g$ and $j$ admit right adjoints $g\dashv g^R$ and $j\dashv j^R$.  Then by composing units and counits one obtains a canonical transformation
    \[
      \alpha: f g^R \to  j^R j f g^R \simeq j^R h g g^R \to j^R g.
    \]
  \item
    Dually, if $g$ and $j$ admit left adjoints $g^L\dashv g$ and $j^L\dashv j$, then there is a canonical transformation
    \[
      \beta: j^L h \to j^L hg g^L \simeq  j^L j f g^L \to f g^L.
    \]
\end{itemize}

\begin{definition}[Adjoinable square]
  We say that the square (\ref{sq:adjoinable}) above is
  \begin{itemize}
    \item
      \emph{vertically right-adjoinable} if $g,j$ admit right adjoints and $\alpha: f g^R \to j^R g$ is an equivalence.
    \item
      \emph{vertically left-adjoinable} if $g,j$ admit left adjoints and $\beta: j^L h \to f g^L$ is an equivalence.
  \end{itemize}
  Horizontally right- or left-adjoinable squares are defined analogously by applying the same criterion to the horizontal adjoints.
\end{definition}

\begin{definition}[{Beck–Chevalley functor, \cite[Definition 3.1]{dauser2025uniquenesssixfunctorformalisms}}]\label{def:bcfun}
  Let $(\mathcal{C},E,I,P)$ be a Nagata set‐up. Define
  \[
    \BCFun(\mathcal{C},E,I,P)\subseteq\Fun(\mathcal{C}^{\op},\Cat_\infty)
  \]
  to be the subcategory of those functors $D: \mathcal{C}^{\op}\to\Cat_\infty$ such that for every cartesian square
  \[
    \begin{tikzcd}
      {X'} & X \\
      Y' & Y
      \arrow["{\bar g}", from=1-1, to=1-2]
      \arrow["{\bar f}"', from=1-1, to=2-1]
      \arrow["{f}"', from=1-2, to=2-2]
      \arrow["{g}", from=2-1, to=2-2]
      \arrow["\lrcorner"{anchor=center, pos=0.125}, draw=none, from=1-1, to=2-2]
    \end{tikzcd}
  \]
  and its image under $D$
  \begin{equation}\label{sq:cartesian-adjoinable-condition}
    \begin{tikzcd}
      {D(X')} & {D(X)} \\
      {D(Y')} & {D(Y)}
      \arrow["{\bar g^*}"', from=1-2, to=1-1]
      \arrow["{\bar f^*}", from=2-1, to=1-1]
      \arrow["{f^*}", from=2-2, to=1-2]
      \arrow["{g^*}"', from=2-2, to=2-1]
    \end{tikzcd}
  \end{equation}
  \begin{enumerate}
    \item (\emph{Proper base‐change})
      if $f\in P$, then $f^*,\bar f^*$ admit right adjoints $f_*,\bar f_*$ respectively, and the square (\ref{sq:cartesian-adjoinable-condition}) is vertically right‐adjointable.
    \item (\emph{Étale base-change})
      If $f\in I$, then $f^*,\bar f^*$ admit left adjoints $f_\natural,\bar f_\natural$ respectively, and the square (\ref{sq:cartesian-adjoinable-condition}) is vertically left‐adjointable.
    \item (\emph{Mixed Beck–Chevalley})
      Whenever $f\in I$ and $g\in P$, the vertical left-adjoint of the square (\ref{sq:cartesian-adjoinable-condition})
      \[\begin{tikzcd}
        {D(X')} & {D(X)} \\
        {D(Y')} & {D(Y)}
        \arrow["{\bar f_\natural}"', from=1-1, to=2-1]
        \arrow["{\bar g^*}"', from=1-2, to=1-1]
        \arrow["{f_\natural}"', from=1-2, to=2-2]
        \arrow["{g^*}"', from=2-2, to=2-1]
      \end{tikzcd}\]
      is horizontally right-adjoinable, i.e., the canonical exchange transformation
      \[
        f_\natural\,g^*\xrightarrow{\simeq}g^*\,\bar f_\natural
      \]
      is an equivalence.
  \end{enumerate}
  A natural transformation $\phi: D\to D'$ lies in $\BCFun(\mathcal{C},E,I,P)$ if the square
  \[\begin{tikzcd}
    {D(Y)} & {D(Y')} \\
    {D(X)} & {D(X')}
    \arrow["{\phi_Y}", from=1-1, to=1-2]
    \arrow["{f^*}"', from=2-1, to=1-1]
    \arrow["{\phi_X}", from=2-1, to=2-2]
    \arrow["{f^*}"', from=2-2, to=1-2]
  \end{tikzcd}\]
  is vertically right-adjoinable if $f$ is in $P$, and vertically left-adjoinable if $f$ is in $I$. Equivalently, for each $i: X\to Y$ in $I$, the exchange transformation $i_! \,\phi_X \to \phi_Y \,i_!$ is an equivalence, and for each $p: X\to Y$ in $P$, the exchange transformation $p_* \,\phi_X \to \phi_Y\,p_*$ is an equivalence.
\end{definition}

\begin{definition}[{\cite[Definition 3.1]{dauser2025uniquenesssixfunctorformalisms}}]
  Let
  \[
    \BCFun^{\lax}(\mathcal{C},E,I,P)\subseteq\BCFun(\mathcal{C}^{\op,\sqcup,\op},E_-,I_-,P_-)
  \]
  be the full subcategory spanned by functors $\mathcal{C}^{\op, \sqcup} \to \Cat_\infty$ that are lax cartesian structures\footnote{See \cite[Notation 2.2, Notation 2.3]{dauser2025uniquenesssixfunctorformalisms} for the relevant notations and details of the lax cartesian structure.}. Let
  \[
    \BCFun^{\lax,\mathrm{L}}(\mathcal{C},E,I,P)\subset\BCFun^{\lax}(\mathcal{C},E,I,P)
  \]
  be the full subcategory spanned by those lax‐cartesian structures $D:\mathcal{C}^{\op,\sqcup}\to\Cat_\infty$ satisfying:
  \begin{enumerate}
    \item
      For every $X\in \mathcal{C}$, the symmetric monoidal $\infty$‐category $D(X)$ is closed.
    \item
      For every $f: X\to Y$ in $\mathcal{C}$, the functor $f^*$ admits a right adjoint $f_*$.
    \item
      For every $p: X\to Y$ in $P$, the functor $p_*$ admits a right adjoint $p^\natural$.
  \end{enumerate}
\end{definition}

\begin{remark}
  The projection formula is implicitly encoded in the lax monoidal structure. Concretely, let $D\in \BCFun^{\lax}(\mathcal{C},E,I,P)$, and let $f: X \to Y$ be a morphism in $I$. Then we have a commutative square
  \[\begin{tikzcd}[sep=large]
    {D(Y) \times D(Y)} & {D(Y)} \\
    {D(X) \times D(Y)} & {D(X)}
    \arrow["\otimes", from=1-1, to=1-2]
    \arrow["{f^* \times \id_{D(Y)}}"', from=1-1, to=2-1]
    \arrow["{f^*}"', from=1-2, to=2-2]
    \arrow["{(-) \otimes f^*(-)}", from=2-1, to=2-2]
  \end{tikzcd}\]
  in which the vertical morphisms are induced by $f\times \id_Y$ and $f$, respectively, both lying in $I_-$. The condition that morphisms in $I$ satisfies the projection formula is equivalent to the requirement that the above diagram is vertically left-adjoinable. An analogous statement applies to the class of morphisms $P$.
\end{remark}

\begin{definition}[{Nagata six-functor formalism, \cite[Definition 2.15]{dauser2025uniquenesssixfunctorformalisms}}]
  Let $(\mathcal{C},E,I,P)$ be a Nagata set‐up. We define
  \[
    \SixFF(\mathcal{C},E,I,P) \subseteq \Fun(\Corr(C, E)^\otimes, \Cat_\infty)
  \]
  to be the full subcategory spanned by the six-functor formalisms that are Nagata, i.e., six-functor formalisms in which every morphism in $P$ is cohomologically proper \cite[Definition 6.10]{scholze2022sixfunctors}, and every morphism in $I$ is cohomologically étale \cite[Definition 6.12]{scholze2022sixfunctors}.
\end{definition}

\begin{theorem}[{\cite[Theorem 3.3]{dauser2025uniquenesssixfunctorformalisms}}]\label{thm:six-functor-as-bcfun}
  Let $(\mathcal{C},E,I,P)$ be a Nagata set‐up.  Restriction along
  \[
    \mathcal{C}^{\op,\sqcup}\to\Corr(\mathcal{C},E)^{\otimes}
  \]
  induces equivalences of $\infty$-categories
  \[
    \BCFun^{\lax,\mathrm{L}}(\mathcal{C},E,I,P)\simeq\SixFF(\mathcal{C},E,I,P).
  \]
\end{theorem}

\begin{remark}
  The inverse of the equivalence $\SixFF(\mathcal{C},E,I,P) \xrightarrow{\simeq} \BCFun^{\lax,\mathrm{L}}(\mathcal{C},E,I,P)$ is given by the Liu–Zheng and Mann construction (\cref{thm:construction-six-functors}).
\end{remark}

With the abstract framework established, we now turn our attention to the specific setting of locally compact Hausdorff spaces relevant to our study.

\begin{convention}
  For the remainder of this section, we fix the source category to be the 1-category $\LCH$, and we take the classes of morphisms $E$, $I$, and $P$ to consist of all morphisms, open immersions, and proper maps, respectively.

  For simplicity, when the context is clear, we omit the classes $E, I, P$ and write $\BCFun^{\lax,\mathrm{L}}(\LCH)$ for $\BCFun^{\lax,\mathrm{L}}(\LCH,E,I,P)$ analogously write $\SixFF(\LCH)$ for $\SixFF(\LCH,E,I,P)$.
\end{convention}

\begin{definition}[Continuous six-functor formalism]\label{def:continuous-six-functor}
  A Nagata six-functor formalism $D\in \SixFF(\LCH)$ taking values in dualizable stable $\infty$-categories is called \emph{continuous} if it satisfies the following properties:

  \begin{enumerate}

    \item (\emph{Canonical descent/Localization})
      \begin{enumerate}
        \item
          $D(\emptyset) = 0$.
        \item
          For each closed immersion $i: Z\hookrightarrow X$ with open complement $j: U\hookrightarrow X$, and the square
          \[
            \begin{tikzcd}
              D(Z) \ar{r}{i_*} \ar{d} & D(X) \ar{d}{j^*}\\
              0 \ar{r}                & D(U)
            \end{tikzcd}
          \]
          is Cartesian in $\Cat_\infty^{\ex}$.
        \item
          For an increasing, filtered union $U = \bigcup_{i\in I} U_i$ of opens in $X$ we have that
          \[
             D(U) \xrightarrow{\simeq} \lim\nolimits^{\Cat_\infty^{\ex}} D(U_i)
          \]
          where the limit is taken along the functors $(f_{ij})^\ast: D(U_j) \to D(U_i)$.
      \end{enumerate}

    \item (\emph{Profinite descent})
      For a cofiltered limit of compact Hausdorff spaces $X = \lim_{i \in I} X_i$, there is an equivalence
      \[
        D(X) \xrightarrow{\simeq} \lim\nolimits^{\Cat_\infty^{\ex}} D(X_i)
      \]
      where the limit is taken along the functors $(f_{ij})_\ast: D(X_i) \to D(X_j)$.

    \item (\emph{Hyperdescent})
      For a hypercomplete locally compact Hausdorff spaces $X$, equivalences in $D(X)$ can be detected stalkwise, i.e., the functor
      \[
        \Big( \prod_{x: \pt \to X} x^* \Big): D(X) \to \prod_{x\in X} D(\pt)
      \]
      reflects equivalences.
  \end{enumerate}
  We denote by $\SixFF(\LCH)^{\cont} \subseteq \SixFF(\LCH)$ the subcategory consisting of continuous six-functor formalisms, where the morphisms are those natural transformations whose individual components preserve colimits.
\end{definition}

\begin{proposition}\label{prop:continuous-six-functor-as-subcategory-of-cosys}
  The category $\SixFF(\LCH)^{\cont}$ can be identified with a subcategory of $\CoSys^{\mathrm{c}}$ consisting of cocomplete coefficient systems $C$, valued in dualizable stable $\infty$-categories, that satisfy the following additional conditions:
  \begin{enumerate}
    \item
      Whenever $p:X\to Y$ is a \emph{proper map} in $\LCH$, the functor $p^*:C(Y)\to C(X)$ admits a right adjoint $p_*$, and:
      \begin{enumerate}
        \item (\emph{Base change})
          For each Cartesian square
          \[\begin{tikzcd}
            {X'} & {Y'} \\
            X & Y
            \arrow["{p'}", from=1-1, to=1-2]
            \arrow["{f'}"', from=1-1, to=2-1]
            \arrow["\lrcorner"{anchor=center, pos=0.125}, draw=none, from=1-1, to=2-2]
            \arrow["f", from=1-2, to=2-2]
            \arrow["p", from=2-1, to=2-2]
          \end{tikzcd}\]
          the canonical exchange transformation $p'_*\,(f')^* \to f^*\,p_*$ is an equivalence.
        \item (\emph{Projection formula})
          The canonical map
          \[
            p_*(p^*(-)\otimes -) \to - \otimes p_*(-)
          \]
          is an equivalence of functors $C(Y)\times C(X)\to C(Y)$.
      \end{enumerate}
    \item
      $C$ satisfies \emph{profinite descent} and \emph{hyperdescent} in the sense of \cref{def:continuous-six-functor}.
  \end{enumerate}
  Moreover, the morphism $\phi$ in this subcategory are required to be compatible with proper pushforwards, i.e., the exchange transformation
  \[
    p_* \,\phi_{X'} \to \phi_X \,p_*
  \]
  is an equivalence for every proper map $p: X \to Y$.
\end{proposition}

\begin{proof}
  This follows from \cref{def:bcfun} and \cref{thm:six-functor-as-bcfun}. Note that for an open subset $U\subseteq X$, localization induces a conservative functor (see \cite[Proposition A.2.19]{calmès2025hermitianktheorystableinftycategories})
  \[
    C(X) \to C(U) \times C(X \setminus U).
  \]
  It follows that the mixed Beck–Chevalley condition required for the $\BCFun$ in \cref{def:bcfun} is satisfied, as noted in \cite[Remark 4.2]{scholze2022sixfunctors}.
\end{proof}

\begin{proposition}\label{prop:sheaf-is-continuous-six-functor}
  The assignment $X \mapsto \Shv(X; \Sp)$ defines a continuous six-functor formalism on locally compact Hausdorff spaces, i.e., $\Shv(-; \Sp)\in\SixFF(\LCH)^{\cont}$.
\end{proposition}

\begin{proof}
  By \cref{prop:sheaf-is-cosys}, $\Shv(-;\Sp)\in\CoSys^{\mathrm{c}}$. Hence, by \cref{prop:continuous-six-functor-as-subcategory-of-cosys}, it suffices to show that the sheaf functor satisfies:
  \begin{enumerate}
    \item the base‐change and projection formulas for proper maps;
    \item profinite descent;
    \item hyperdescent.
  \end{enumerate}
  The base change and projection formula are established in \cite[Proposition 6.9, Proposition 6.12]{volpe2023operationstopology}, profinite descent is given by \cref{prop:sheaf-profinite-descent}, and hyperdescent holds by definition.
\end{proof}

\begin{lemma}\label{lem:compact-supp-cohomology-recollement}
  Let $D\in \SixFF(\LCH)^{\cont}$ be a continuous six-functor formulism, and $X$ be a locally compact Hausdorff space. For an open immersion $i: U \hookrightarrow X$ and the closed complement $j: V = X \setminus U \hookrightarrow X$, there is an exact sequence in $D(\pt)$
  \[
    \pi_{U,!} \pi_U^* 1_{D(\pt)} \to \pi_{X,!} \pi_X^* 1_{D(\pt)} \to \pi_{V,!} \pi_V^* 1_{D(\pt)}.
  \]
\end{lemma}

\begin{proof}
  Consider the recollement arising from the localization sequence (see \cref{rem:verdier-sequence-from-localization}), which yields the exact sequence
  \[
    i_! i^* \to \id_{D(X)} \to j_* j^*,
  \]
  (see \cite[Remark A.2.9]{calmès2025hermitianktheorystableinftycategories}). Applying the exact functors $\pi_{X,!}$ and $\pi_X^*$ to the above exact sequence and using $j_* = j_!$ since $j$ is proper, we obtain the desired exact sequence
  \[
    \pi_{U,!} \pi_U^* 1_{D(\pt)} \to \pi_{X,!} \pi_X^* 1_{D(\pt)} \to \pi_{V,!} \pi_V^* 1_{D(\pt)}.
  \]
\end{proof}

\begin{lemma}\label{lem:compact-supp-cohomology-cosheaf}
  In the same setting as in \cref{lem:compact-supp-cohomology-recollement}, the functor $\mathcal{F}: \Open(X) \to \Sp$
  \[
    \big( U \xrightarrow{\pi_U} \pt \big) \mapsto \pi_{U,!} \pi_U^* 1_{D(\pt)}.
  \]
  defines a cosheaf valued in $D(\pt)$ on $X$.
\end{lemma}

\begin{proof}
  We verify the cosheaf axioms:
  \begin{enumerate}
    \item
      By construction, the functor is reduced, i.e., it sends the empty space to a zero object.

    \item
      For any open subsets $U, V \subseteq X$, by \cref{lem:compact-supp-cohomology-recollement}, there is a commutative diagram in $D(\pt)$ with exact rows
      \[\begin{tikzcd}[column sep=small]
        {\pi_{U\cap V,!} \pi_{U\cap V}^* 1_{D(U\cap V)}} & {\pi_{U,!} \pi_{U}^* 1_{D(U)}} & {\pi_{U\setminus V,!} \pi_{U\setminus V}^* 1_{D(U\setminus V)}} \\
        {\pi_{V,!} \pi_{V}^* 1_{D(V)}} & {\pi_{U\cup V,!} \pi_{U\cup V}^* 1_{D(U\cup V)}} & {\pi_{U\setminus V,!} \pi_{U\setminus V}^* 1_{D(U\setminus V)}}
        \arrow[from=1-1, to=1-2]
        \arrow[from=1-1, to=2-1]
        \arrow[from=1-2, to=1-3]
        \arrow[from=1-2, to=2-2]
        \arrow["{=}", from=1-3, to=2-3]
        \arrow[from=2-1, to=2-2]
        \arrow[from=2-2, to=2-3]
      \end{tikzcd}\]
      It follows that the left square is a pushout in $D(\pt)$.

    \item
      For a filtered union $U = \bigcup_{i\in I} U_i$ of opens in $X$, we have
      \[
        \begin{aligned}
          \pi_{U,!} \pi_U^* 1_{D(\pt)} \simeq \pi_{U,!} 1_{D(U)} &\simeq \pi_{U,!} \colim \iota_{i, !} 1_{D(U_i)} \\
          & \simeq \colim \pi_{U,!}\iota_{i, !} 1_{D(U_i)} \\
          & \simeq \colim \pi_{U_i,!} 1_{D(U_i)} \\
          & \simeq \colim \pi_{U_i,!} \pi_{U_i}^* 1_{D(\pt)}. \\
        \end{aligned}
      \]
      Here, the second equivalence comes from the localization axiom\footnote{The colimit version of the localization axiom is obtained by passing to the left adjoint.}
      \[
         D(U) \xrightarrow{\simeq} \colim\nolimits^{\Cat_\infty^{\ex}} D(U_i)
      \]
      where the filtered colimit is taken along the extension functoriality. \qedhere
  \end{enumerate}
\end{proof}

\begin{lemma}\label{lem:compact-supp-cohomology-agree-sp}
  Let $D\in \SixFF(\LCH)^{\cont}$ be a continuous six-functor formulism. Suppose that $D(\pt)\simeq\Sp$. Then the functor
  \[
    \LCH \to \Sp: X \mapsto \pi_{X,!} \pi_X^* 1_{D(\pt)}
  \]
  is equivalent to the compactly supported sheaf cohomology with $\S \simeq 1_{\Sp}$ coefficients, where $\pi_X: X \to \pt$ denotes the unique map to a point which induces $\pi_X^*: D(\pt) \to D(X)$ and $\pi_{X,!}: D(X) \to D(\pt)$.
\end{lemma}

\begin{proof}
  Let $X$ be a locally compact Hausdorff space. By \cref{lem:compact-supp-cohomology-cosheaf}, one obtains a spectral cosheaf $\mathcal{F}: \Open(X) \to \Sp$ on $X$, defined by
  \[
    \big( U \xrightarrow{\pi_U} \pt \big) \mapsto \pi_{U,!} \pi_U^* 1_{D(\pt)}.
  \]
  Applying Verdier duality (\cref{thm:verdier-duality}),
  \[
    \D: \CoShv(X; \Sp) \simeq \Shv(X; \Sp),
  \]
  we obtain a sheaf $\D(\mathcal{F})$. For any compact subset $K\subseteq X$, Verdier duality gives
  \[
    \D(\mathcal{F})(K) = \mathcal{F}(X)/\mathcal{F}(X \setminus K).
  \]
  where the value of a sheaf on a subset denotes the global sections of its restriction to that subset. To describe $\D(\mathcal{F})(K)$, consider the exact sequence in $\Sp \simeq D(\pt)$ given by \cref{lem:compact-supp-cohomology-recollement}
  \[
    \pi_{X\setminus K,!} \pi_{X\setminus K}^* 1_{D(\pt)} \to \pi_{X,!} \pi_X^* 1_{D(\pt)} \to \pi_{K,!} \pi_K^* 1_{D(\pt)}.
  \]
  It follows that
  \[
    \D(\mathcal{F})(K) = \mathcal{F}(X)/\mathcal{F}(X \setminus K) \simeq \pi_{K,!} \pi_K^* 1_{D(\pt)}.
  \]
  Since $K$ is compact, $\pi_K$ is proper and $\pi_{K,!} = \pi_{K,*}$, the unit map
  \[
    \S \simeq 1_{D(\pt)} \to \pi_{K,*} \pi_K^* 1_{D(\pt)} \simeq \pi_{K,!} \pi_K^* 1_{D(\pt)} = \D(\mathcal{F})(K)
  \]
  adjoins over to a morphism $\underline{\S} \to \D(\mathcal{F})|_{K}$. As $X$ is locally compact, these local morphisms glue into a morphism of sheaves $\underline{\S} \to \D(\mathcal{F})$.

  To complete the proof, we must show that the morphism $\underline{\S} \to \D(\mathcal{F})$ is an equivalence. Indeed, under Verdier duality, a sheaf $\mathcal{G}$ corresponds to the cosheaf $\D^{-1}(\mathcal{G})$ whose value on $U$ is the compactly supported sections of $\mathcal{G}$ on $U$, which then gives the desired description of $\mathcal{F}$.

  It remains to verify that $\underline{\S} \to \D(\mathcal{F})$ is an equivalence. As sheaves on locally compact spaces are determined by their values on compact subsets, and constant sheaves are stable under pullback, it suffices to check the equivalence when $X$ is compact. Moreover, if the equivalence holds on a compact space, it holds on each of its closed subspaces. Now every compact Hausdorff space embeds into a Hilbert cube $[0,1]^I$ for some set $I$, via the canonical double dual map $X \to [0,1]^I$, where $I = C^0(X, [0,1])$. This map is injective by Urysohn’s Lemma and defines an embedding. Thus, it suffices to prove the claim for $[0,1]^I$.

  The Hilbert cube is an inverse limit of finite-dimensional cubes $[0,1]^n$. By \cref{prop:section-profinite-descent} together with the profinite descent property of the continuous six-functor formalism, it is enough to verify the claim for each $[0,1]^n$. These spaces are hypercomplete, and by the hyperdescent property, the equivalence of sheaves can be checked on stalks. Since both $\underline{\S}$ and $\D(\mathcal{F})$ have stalks equivalent to $\S$, the morphism is indeed an equivalence.
\end{proof}

\begin{lemma}\label{lem:compact-supp-cohomology-agree}
  Let $D\in \SixFF(\LCH)^{\cont}$ be a continuous six-functor formalism. Then the functor
  \[
    \LCH \to \Sp: X \mapsto \pi_{X,!} \pi_X^* 1_{D(\pt)}
  \]
  is equivalent to the compactly supported sections functor $X \mapsto \Gamma_{\mathrm{c}}(X, \underline{1_{D(\pt)}})$ where $\pi_X: X \to \pt$ denotes the unique map to a point and sections are taken in the sheaf $\Shv(-; D(\pt))$.
\end{lemma}

\begin{proof}
  The argument follows identically to that of \cref{lem:compact-supp-cohomology-agree-sp} with the stalkwise check following directly from \cref{lem:stalkwise-check}.
\end{proof}

\begin{theorem}\label{thm:sheaf-initial-six-functor}
  The object $\Shv(-; \Sp) \in \SixFF(\LCH)^{\cont}$ is initial.
\end{theorem}

\begin{proof}
  By \cref{prop:sheaf-is-continuous-six-functor} we have $\Shv(-; \Sp) \in \SixFF(\LCH)^{\cont}$. Note that $\Shv(-;\Sp)\in\CoSys^{\mathrm{c}}$ is initial by \cref{thm:sheaf-initial-cosys}. Therefore, by \cref{prop:continuous-six-functor-as-subcategory-of-cosys}, it suffices to show that for every $D\in \SixFF(\LCH)^{\cont}$ the unique morphism $[-]: \Shv(-; \Sp) \to D$ is compatible with proper pushforwards, i.e., the following diagram commutes
  \[\begin{tikzcd}[column sep=3em]
    {\Shv(X; \Sp)} & {D(X)} \\
    {\Shv(Y; \Sp)} & {D(Y)}
    \arrow["{[-]}"{description}, from=1-1, to=1-2]
    \arrow["{p_*}", from=1-1, to=2-1]
    \arrow["{p_*}", from=1-2, to=2-2]
    \arrow["{[-]}"{description}, from=2-1, to=2-2]
  \end{tikzcd}\]
  for every proper map $p: X \to Y$. Since $\Shv(X; \Sp)$ is generated by representable spectral sheaves under filtered colimits and both $[-]$ and $p_*$ (as the left adjoint to $p^!$ by properness) preserves colimits, it suffices to verify the commutativity of the above diagram on these generators. Note that the spectral sheaf represented by an open immersion $i: U \to X$ can be expressed as
  \[
    \Sigma^\infty_+ \Map(-, U) \simeq i_! \pi_U^* 1_{\Sp}
  \]
  where $\pi_U: U \to \pt$ denotes the unique map to a point.

  Using the localization property, it suffices to verify the desired equivalence on each compact subset of $Y$, allowing us to reduce to the case where $Y$ is compact. As in the proof of \cref{lem:compact-supp-cohomology-agree-sp}, we can embed $Y$ into a larger space, e.g., embedding into the Hilbert cube $[0,1]^{C^0(X, [0, 1])}$ via the canonical double dual map. Therefore, using the profinite descent property, we may also assume that $Y$ is hypercomplete. Then for any $y: \pt \to Y$, let $X_y = X \times_Y \{ y \}$ and $U_y = U \times_Y \{ y \}$. Consider the following commutative diagram
  \[\begin{tikzcd}[row sep=1.5em, column sep=1em]
    {\Shv(U_y; \Sp)} && {\Shv(U; \Sp)} \\
    & {D(U_y)} && {D(U)} \\
    {\Shv(X_y; \Sp)} && {\Shv(X; \Sp)} \\
    & {D(X_y)} && {D(X)} \\
    {\Shv(\pt; \Sp)} && {\Shv(Y; \Sp)} \\
    & {D(\pt)} && {D(Y)}
    \arrow[from=1-1, to=2-2]
    \arrow[from=1-1, to=3-1]
    \arrow[from=1-3, to=1-1]
    \arrow[from=1-3, to=2-4]
    \arrow["{i_!}"{pos=0.66}, from=1-3, to=3-3]
    \arrow["{i_!}"{pos=0.66}, from=2-4, to=4-4]
    \arrow[from=3-1, to=4-2]
    \arrow[from=3-1, to=5-1]
    \arrow[from=3-3, to=3-1]
    \arrow[from=3-3, to=4-4]
    \arrow["{p_*}"{pos=0.66}, from=3-3, to=5-3]
    \arrow["{p_*}"{pos=0.66}, from=4-4, to=6-4]
    \arrow[from=5-1, to=6-2]
    \arrow["{y^*}"{pos=0.33}, from=5-3, to=5-1]
    \arrow[from=5-3, to=6-4]
    \arrow["{y^*}"{pos=0.33}, from=6-4, to=6-2]
    \arrow[crossing over, from=2-2, to=4-2]
    \arrow[crossing over, from=2-4, to=2-2]
    \arrow[crossing over, from=4-2, to=6-2]
    \arrow[crossing over, from=4-4, to=4-2]
  \end{tikzcd}\]
  by the open base change and proper base change, it suffices to verify the commutativity of the left face. Hence we may reduce further to the case $Y = \pt$.

  In this setting, using the fact that $p$ is proper, so that $p_* = p_!$, and that the morphism $[-]$ preserves units and is compatible with both $i_!$ and $\pi_U^*$, it remains to show the equivalence
  \[
    [\pi_{U,!} \pi_U^* 1_{\Sp}] \simeq \pi_{U,!} \pi_U^* 1_{D(\pt)}.
  \]
  This follows immediately from \cref{cor:factor-initial-cosys-map} and \cref{lem:compact-supp-cohomology-agree}.
\end{proof}

\begin{proposition}\label{prop:prop:free-forgetful-six-functor-adjunction}
  Let $U: \SixFF(\LCH)^{\cont} \to \CAlg(\Cat^{\dual,\mathrm{L}}_{\infty})$ denote the forgetful functor given by evaluation at the point. Then $U$ is right adjoint to the functor
  \[
    F: \CAlg(\Cat^{\dual,\mathrm{L}}_{\infty}) \to \SixFF(\LCH)^{\cont},\quad \mathcal{C} \mapsto \Shv(X; \mathcal{C}).
  \]
  In other words, the assignment $X \mapsto \Shv(X; \mathcal{C})$ exhibits the free continuous six-functor formalism generated by the symmetric monoidal dualizable stable $\infty$-category $\mathcal{C}$.
\end{proposition}

\begin{proof}
  Let $D\in \SixFF(\LCH)^{\cont}$. By \cref{prop:free-forgetful-cosys-adjunction}, it remains only to verify that the counit
  \[
    \epsilon: \Shv(-; D(\pt)) \to D
  \]
  is compatible with proper pushforwards. But the same reasoning applies as in the proof of \cref{prop:free-forgetful-cosys-adjunction}, once one replaces the $i_!$ for an open immersion $i$ by the $p_*$ for a proper morphism $p$.
\end{proof}

\begin{corollary}\label{cor:factor-initial-six-functor-map}
  Let $D\in \SixFF(\LCH)^{\cont}$ and denote $\mathcal{D} = D(\pt)$. The unique morphism $[-]: \Shv(-; \Sp) \to D$ factors through the counit $\epsilon: \Shv(-; \mathcal{D}) \to D$.
  \[\begin{tikzcd}[sep=large]
    {\Shv(-; \Sp) } & D \\
    {\Shv(-; \mathcal{D}) }
    \arrow["{[-]}"{description}, from=1-1, to=1-2]
    \arrow["{[-]}"{description}, from=1-1, to=2-1]
    \arrow["\epsilon"', from=2-1, to=1-2]
  \end{tikzcd}\]
\end{corollary}

\begin{corollary}\label{cor:free-six-functor-over-base}
  The object $\Shv(-; \mathcal{B})$ is initial in the category $\SixFF(\LCH)^{\cont}_{\mathcal{B}}$, where $\SixFF(\LCH)^{\cont}_{\mathcal{B}}$ denotes the full subcategory of $\SixFF(\LCH)^{\cont}$ consisting of continuous six-functor formalisms valued in algebras over the base $\mathcal{B} \in \CAlg(\Cat^{\dual, \mathrm{L}}_{\infty})$.
\end{corollary}

We conclude this section by stating a series of consequences of the universal property of sheaves for (co)homology theories derived from the continuous six-functor formalism.

\begin{proposition}\label{prop:cohomology-equivalence}
  Let $D\in \SixFF(\LCH)^{\cont}$ and denote $\mathcal{D} = D(\pt)$. Consider the counit
  \[
    \epsilon: \Shv(-; \mathcal{D}) \to D
  \]
  of the adjunction $F: \CAlg(\Cat^{\dual,\mathrm{L}}_{\infty}) \;\substack{\rightarrow \\[-0.5ex] \leftarrow}\; \SixFF(\LCH)^{\cont}: U$. Then $\epsilon$ is a cohomological equivalence: for any object $E \in \mathcal{D}$, the associated (co)homology functors
  \begin{itemize}
    \item (cohomology functor)
      \[
        \LCH^{\op} \to \mathcal{D}: X\mapsto \pi_{X,*} \pi_X^* E
      \]
    \item (compactly supported cohomology functor)
      \[
        \LCH \to \mathcal{D}: X\mapsto \pi_{X,!} \pi_X^* E;
      \]
    \item (homology functor)
      \[
        \LCH \to \mathcal{D}: X\mapsto \pi_{X,!} \pi_X^! E;
      \]
    \item (locally finite homology functor)
      \[
        \LCH \to \mathcal{D}: X\mapsto \pi_{X,*} \pi_X^! E.
      \]
  \end{itemize}
  are equivalent to the same construction carried out on the sheaf $\Shv(-; \mathcal{D})$. In particular, the (co)homology theories that are intrinsically defined in any continuous six-functor formulism coincide with the standard sheaf (co)homology theories.
\end{proposition}

\Cref{lem:compact-supp-cohomology-agree} establishes the case of compactly supported cohomology with the unit coefficient, which plays a key role in the proof of the initiality of the sheaf (\cref{thm:sheaf-initial-six-functor}). Building on this, \cref{prop:cohomology-equivalence} extends the result of \cref{lem:compact-supp-cohomology-agree} to the remaining three (co)homology functors.

\begin{proof}
  The statement for the cohomology functor follows from the following commutative diagram
  \[\begin{tikzcd}[sep=large]
    {\Shv(X; \mathcal{D})} & {D(X)} \\
    {\Shv(\pt; \mathcal{D})} & {\mathcal{D}}
    \arrow[from=1-1, to=1-2]
    \arrow["{\pi_{X,*}}"', shift right, from=1-1, to=2-1]
    \arrow["{\pi_{X,*}}"', shift right, from=1-2, to=2-2]
    \arrow["{\pi^*_X}"', shift right, from=2-1, to=1-1]
    \arrow["\simeq"', from=2-1, to=2-2]
    \arrow["{\pi^*_X}"', shift right, from=2-2, to=1-2]
  \end{tikzcd}\]
  as the canonical morphism $\Shv(-; \mathcal{D}) \to D$ is a morphism of continuous six-functor formalisms. Similar reasoning applies to the remaining three functors.
\end{proof}

In the framework of six-functor formalisms, Poincaré duality can be interpreted as the statement that for a \emph{smooth} morphism $f: X \to Y$, the right adjoint $f^!$ of $f_!$ exists and agrees with $f^*$ up to a twist. This perspective motivates the following abstract definition.

\begin{definition}[{\cite[Definition 5.1]{scholze2022sixfunctors}}]
  Let $D\in \SixFF(\LCH)^{\cont}$ and $f: X\to Y$ be a morphism in $\LCH$. We say that $f$ is ($D$-)\emph{cohomologically smooth} if all of the following conditions hold:
  \begin{enumerate}
    \item
      The functor $f_!$ admits a right adjoint $f^!$, and the natural transformation
      \[
        f^!(1_Y) \otimes f^*(-) \to f^!(-)
      \]
      of functors $D(Y) \to D(X)$ is an equivalence.
    \item
      The object $f^!(1_Y)\in D(Y)$ is $\otimes$-invertible.
    \item
      The properties (1) and (2) are stable under base change, i.e., for every map $g: Y'\to Y$, form the Cartesian square
      \[\begin{tikzcd}
        {X'} & X \\
        {Y'} & Y
        \arrow["{g'}", from=1-1, to=1-2]
        \arrow["{f'}", from=1-1, to=2-1]
        \arrow["\lrcorner"{anchor=center, pos=0.125}, draw=none, from=1-1, to=2-2]
        \arrow["f", from=1-2, to=2-2]
        \arrow["g", from=2-1, to=2-2]
      \end{tikzcd}\]
      then:
      \begin{itemize}
        \item
          $f$ satisfies the properties (1) and (2) as well;
        \item
          the canonical map
          \[
            g'^* f^! (1_Y) \to f'^! (1_{Y'})
          \]
          is an equivalence.
      \end{itemize}
  \end{enumerate}
  Here the natural transformation
  \[
    f^!(1_Y) \otimes f^*(-) \to f^!(-)
  \]
  is adjoint to
  \[
    f_!(f^!(1_Y) \otimes f^*(-)) \simeq f_!f^!(1_Y) \otimes (-) \to (-)
  \]
  using the projection formula and the counit map $f_! f^! (1_Y) \to 1_Y$, and $g'^* f^! (1_Y) \to f'^! (1_{Y'})$ is adjoint to
  \[
    f'_! g'^* f^! (1_Y) \simeq g^* f_! f^! (1_Y) \to g^*(1_Y) = 1_{Y'}
  \]
  using the base change and the same counit map.
\end{definition}

\begin{convention}
  When $D = \Shv(-; \Sp)$, We omit the explicit notation of $\Shv(-;\Sp)$ and simply write cohomologically smooth. This convention is justified by \cref{prop:universal-cohomologically-smooth} below.
\end{convention}

If $\pi_X: X \to \pt$ is cohomologically smooth and $X$ is compact, let $E\in\Sp$ and write $\omega_X = \pi_X^!(\S)$ for the dualizing complex of $X$, one obtain
\[
  \pi_{X,!} \pi_X^! E \simeq \pi_{X,*}(\pi_X^!(\S) \otimes \pi_X^*(E)) \simeq \pi_{X,*}(\omega_X \otimes \pi_X^*(E)),
\]
which can be interpreted as the isomorphism $H_*(A; E) \simeq H^*(X; E\otimes \omega_X)$. In particular, when $X$ is an orientable compact $n$-manifold and one takes the coefficient $\Z$, the dualizing complex $\omega_X$ is simply a shift by $n$, recovering the classical statement of Poincaré duality.

\begin{proposition}\label{prop:R1-cohomologically-smooth}
  The map $\pi_{\R}: \R \to \pt$ is cohomologically smooth and the dualizing complex of $\R$ is a shift by $1$, i.e., $\omega_{\R} = \pi_X^!(\S) \simeq \S[1]$.
\end{proposition}

\begin{proof}
  This statement is established for the derived category of abelian sheaves in \cite[Proposition 5.10, Proposition 7.9]{scholze2022sixfunctors}. For the case of spectral sheaves, see \cite[Proposition 4.6.15, Corollary 4.6.20]{krause2024sheaves}.
\end{proof}

\begin{proposition}\label{prop:universal-cohomologically-smooth}
  Let $D\in \SixFF(\LCH)^{\cont}$ and $f: X\to Y$ be a morphism in $\LCH$. If $f$ is cohomologically smooth, then $f$ is ($D$-)cohomologically smooth.
\end{proposition}

\begin{proof}
  The unique morphism $[-]: \Shv(-; \mathcal{C}) \to D$ of continuous six-functor formalisms induces a functor of $(\infty, 2)$-category
  \[
    \mathrm{LZ}_{\Shv(-;\Sp)} \to \mathrm{LZ}_{D}
  \]
  which carries the commutativity datum of \cite[Theorem 5.5]{scholze2022sixfunctors} that witness the cohomological smoothness encoded in $\mathrm{LZ}_{\Shv(-;\Sp)}$ to $\mathrm{LZ}_{D}$. Then the result follows directly from \cite[Theorem 5.5]{scholze2022sixfunctors}.
\end{proof}

\begin{remark}
  By the reasoning above, morphisms of six-functor formalisms transfer cohomological smoothness in general. In particular, the analogue of \cref{prop:universal-cohomologically-smooth} also holds for $\mathrm{SH}$ and for the category $\SixFF^{DG}$ in which it is initial \cite[Corollary 3.6]{dauser2025uniquenesssixfunctorformalisms}.
\end{remark}

As a consequence of the preceding propositions, continuous six-functor formalisms are homotopy invariant. This corresponds to the $\mathbb{A}^1$-homotopy invariance of the coefficient system in motivic homotopy theory (see \cite[Definition 7.5]{Drew_2022}). In contrast, in the topological setting, homotopy invariance arises from the universal property of sheaves, rather than being imposed as part of the definition.

\begin{theorem}\label{thm:homotopy-invariant-six-functor}
  Let $D\in \SixFF(\LCH)^{\cont}$. Then $D$ is homotopy invariant, i.e., for every locally compact Hausdorff space $X$, the projection $\pi: X\times\R \to X$ induces a fully faithful functor
  \[
    \pi^*: D(X) \to D(X\times\R).
  \]
\end{theorem}

\begin{proof}
  It is equivalent to show the unit $\id_{D(X)} \to \pi_* \pi^*$ is an equivalence. Consider the pullback square
  \[\begin{tikzcd}
    {X\times \R^1} & {\R^1} \\
    X & \pt
    \arrow["p", from=1-1, to=1-2]
    \arrow["\pi", from=1-1, to=2-1]
    \arrow["\lrcorner"{anchor=center, pos=0.125}, draw=none, from=1-1, to=2-2]
    \arrow["{\pi'}", from=1-2, to=2-2]
    \arrow["{p'}", from=2-1, to=2-2]
  \end{tikzcd}\]
  By the stability of cohomological smoothness under base change, together with \cref{prop:R1-cohomologically-smooth,prop:universal-cohomologically-smooth}, the projection $\pi: X\times\R \to X$ is cohomologically smooth. Moreover, we have
  \[
    \pi^! (-) \simeq \pi^! 1_{D(X)} \otimes \pi^* (-) \simeq p^* \pi'^! 1_{D(\pt)} \otimes \pi^* (-) \simeq \pi^* (-)[1].
  \]
  On the other hand, by the projection formula,
  \[
    \pi_! 1_{D(X\times\R)} \simeq \pi_! p^* 1_{D(\R)} \simeq p'^* \pi'_! 1_{D(\R)},
  \]
  and by \cref{prop:cohomology-equivalence},
  \[
    \pi'_! 1_{D(\R)} \simeq \pi'_! \pi'^* 1_{D(\pt)} \simeq \Gamma_{\mathrm{c}}(\R, \underline{\S})\otimes 1_{D(\pt)} \simeq  1_{D(\pt)}[-1].
  \]
  Hence,
  \[
    \pi_! 1_{D(X\times\R)} \simeq  1_{D(X)}[-1].
  \]
  Now for any $F\in D(X)$, we have
  \[\begin{aligned}
    F \simeq \underline{\Hom}(1_{D(X)}[-1], F[-1]) &\simeq \underline{\Hom}(\pi_! 1_{D(X\times\R)}, F[-1])\\
    &\overset{(*)}{\simeq} \pi_* \underline{\Hom}(1_{D(X\times\R)}, \pi^! F[-1]) \label{eq:lineA}\\
    &\simeq \pi_* \pi^! F[-1]\\
    &\simeq \pi_* \pi^* F
  \end{aligned}\]
  where the equivalence $(*)$ follows from the tensor-Hom adjunction and the projection formula. This completes the proof.
\end{proof} 
\section{Localizing invariant} \label{sec:localizing-invariant}

In this section, we study the values of (continuous) localizing invariants on continuous six-functor formalisms. Localizing invariants were originally introduced on $\Cat^{\ex}_\infty$ by Blumberg–Gepner–Tabuada \cite{Blumberg_2013} to characterize higher algebraic $K$-theory, and have more recently been extended to $\Cat^{\dual}_\infty$ by Efimov \cite{efimov2025ktheorylocalizinginvariantslarge}.

We begin by recalling the definitions and results that will be needed below. 

\begin{definition}
  A \emph{Verdier sequence} in $\Cat^{\ex}_\infty$ is a sequence with vanishing composite
  \[
    \mathcal{C} \xrightarrow{i} \mathcal{D} \xrightarrow{p} \mathcal{E}
  \]
  which is a short exact sequence (i.e., both a fiber and cofiber sequence) in $\Cat^{\ex}_\infty$. The exact functor $i$ is called a \emph{Verdier inclusion}, and $p$ a \emph{Verdier projection}.

  We call a Verdier sequence \emph{left-split} (resp. \emph{right-split}) if $p$ admits a left (resp. right) adjoint, and \emph{split} if both adjoints exist.
\end{definition}

To determine whether a functor $f:\mathcal{C} \to \mathcal{D}$ fits into a Verdier sequence, we have the following criterion for Verdier projections and inclusions, adapted from \cite[Corollary A.1.7, Proposition A.1.9]{calmès2025hermitianktheorystableinftycategories}.

\begin{proposition}
  Let $f:\mathcal{C} \to \mathcal{D}$ be an exact functor of stable $\infty$-categories. Then
  \begin{itemize}
    \item
      $f$ is a Verdier inclusion if and only if it is fully faithful and its essential image in $\mathcal{D}$ is closed under retracts.
    \item
      $f$ is a Verdier projection if and only if it is a Dwyer-Kan localization, i.e., there is a collection of morphisms $W$ in $\mathcal{C}$ such that $f^* : \Fun(\mathcal{D}, \mathcal{E}) \to \Fun(\mathcal{C}, \mathcal{E})$ is fully-faithful for any $\infty$-category $\mathcal{E}$ with essential image functors $\mathcal{C} \to \mathcal{E}$ which invert $W$.
  \end{itemize}
\end{proposition}

\begin{definition}[Localizing invariant]
  Let $\mathcal{E}$ be a stable $\infty$-category. A \emph{localizing invariant} with values in $\mathcal{E}$ is a functor $F: \Cat^\ex_\infty \to \mathcal{E}$ satisfying the following properties:
  \begin{enumerate}
    \item
      $F(*) = 0$.
    \item
      $F$ sends Verdier sequences to exact sequences in $\mathcal{E}$.
    \item
      $F$ is invariant under idempotent completion.
  \end{enumerate}
  We say that $F$ is \emph{finitary} if it additionally preserves filtered colimits. We write
  \[
    \Loc_\omega(\mathcal{E}) \subseteq \Loc(\mathcal{E}) \subseteq \Fun(\Cat^\ex_\infty, \mathcal{E}).
  \]
  for the full subcategories of finitary and of all localizing invariants, respectively.
\end{definition}

\begin{remark}
  Due to idempotence invariance, localizing invariants are determined by their values on idempotent-complete stable $\infty$-categories, i.e. on $\Cat^{\perf}_\infty$. Moreover, \cite[Proposition 3.12]{Saunier_2023} implies that, upon restricting to the full subcategory $\Cat^{\perf}_\infty$, one recovers exactly the notion of localizing invariant considered in \cite[\Sec 4]{efimov2025ktheorylocalizinginvariantslarge}.
\end{remark}

\begin{remark}
  Let
  \[
    \mathcal{C} \xrightarrow{i} \mathcal{D} \xrightarrow{p} \mathcal{E}
  \]
  be a split Verdier sequence. Then, by \cite[Lemma A.2.8]{calmès2025hermitianktheorystableinftycategories}, $i$ is fully faithful and admits both left and right adjoints $i^L \dashv i \dashv i^R$, and $p$ admits both left and right adjoints $p^L \dashv p \dashv p^R$ which are both fully faithful. Moreover, the sequence
  \[
    \mathcal{C} \xleftarrow{i^L} \mathcal{D} \xleftarrow{p^L} \mathcal{E}
  \]
  formed by passing to left adjoints is a right split Verdier sequence and the same holds for passing to right adjoints. As a consequence, localizing invariants take split Verdier sequences to split exact sequences.
\end{remark}

\begin{remark}
  The localizing invariance implies the additivity, i.e., for any localizing invariant $F: \Cat^\ex_\infty \to \mathcal{E}$, there is a natural equivalence
  \[
    F(\Fun(\Delta^1, \mathcal{C})) \simeq F(\mathcal{C}) \times F(\mathcal{C}).
  \]
  It follows from the split Verdier sequence \cite[Example A.2.13]{calmès2025hermitianktheorystableinftycategories}
  \[
    \mathcal{C} \xrightarrow{c\mapsto \id_{\mathcal{C}}} \Fun(\Delta^1, \mathcal{C}) \xrightarrow{\mathrm{cofib}} \mathcal{C}.
  \]
\end{remark}

\begin{example}
  (Non-connective) algebraic $K$-theory is a finitary localizing invariant (see \cite[Theorem 1.3]{Blumberg_2013}).
\end{example}

\begin{example}
  Topological Hochschild homology $\THH$ and topological cyclic homology $\TC$ are localizing invariants (see \cite[Proposition 10.2]{Blumberg_2013}\cite[Theorem 3.15]{Saunier_2023}). In addition, since $\THH$ commutes with filtered colimits, it is in fact a finitary localizing invariant. By contrast, $\TC$ does not preserves filtered colimits\footnote{Thanks to Liam Keenan for pointing this out.}. Indeed, there is a canonical morphism
  \[
    \TC(\mathcal{C})[1/p] \to \TC(\mathcal{C}[1/p])
  \]
  for a prime $p\in\Z$, which would be an equivalence if $\TC$ was finitary. Taking $\mathcal{C} = D^{\perf}(\F_p)$, the full subcategory of perfect complexes in $D(\F_p)$, one sees that $\TC(\F_p)[1/p]$ is non-trivial as $\TC_\ast(\F_p) \simeq \Z_p[\epsilon]/\epsilon^2$ with $|\epsilon| = 1$ (see \cite[Corollary IV.4.10]{nikolaus2018topologicalcyclichomology}).
\end{example}

Localizing invariants can only take nontrivial values on ``small'' stable $\infty$-categories. Indeed, they necessarily vanish on any stable $\infty$-category admitting countable colimits. It is a consequence of the Eilenberg swindle. Let $k: \Cat^\ex_\infty \to \mathcal{E}$ be a localizing invariant, and let $F,G:\mathcal{D} \to \mathcal{C}$ be two exact functors. By the additivity
\[
  k(F\sqcup G) \simeq k(F) + k(G).
\]
Assume $\mathcal{C}$ has countable colimits and take $\mathcal{D} = \mathcal{C}$. Consider the endofunctor
\[
  F:\mathcal{C} \to \mathcal{C} \quad X\mapsto \coprod_{\N} X.
\]
Since $F\sqcup \id_{\mathcal{C}} \simeq F$, applying $k$ yields $k(F) \simeq k(F\sqcup \id_{\mathcal{C}}) \simeq k(F) + \id_{\mathcal{E}}$ which forces $\id_{\mathcal{E}}=0$ on $k(\mathcal{C})$. Hence $k(\mathcal{C})$ is trivial.

To extend the theory to “large’’ stable $\infty$-categories, we now recall Efimov’s notion of the continuous localizing invariant for dualizable stable $\infty$-categories. Intuitively, continuity means a localizing invariant is determined by the behavior on compact objects and commutes with passage to $\Ind$-completions. Formally, one obtains a commutative diagram
\[\begin{tikzcd}
  {\Cat^{\ex}_\infty} & {\mathcal{E}} \\
  {\Cat^{\dual}_\infty}
  \arrow["F", from=1-1, to=1-2]
  \arrow["\Ind"', from=1-1, to=2-1]
  \arrow["{F^{\cont}}"', from=2-1, to=1-2]
\end{tikzcd}\]
in which $F: \Cat^\ex_\infty \to \mathcal{E}$ is a localizing invariant and $F^\cont: \Cat^\dual_\infty \to \mathcal{E}$ denotes its (unique) continuous extension. The notion of a continuous localizing invariant $\Cat^\dual_\infty \to \mathcal{E}$ is defined analogously to the case for $\Cat^\ex_\infty$.

\begin{definition}[Continuous localizing invariant]
  Let $\mathcal{E}$ be a stable $\infty$-category. A \emph{continuous localizing invariant} with values in $\mathcal{E}$ is a functor $F: \Cat^\dual_\infty \to \mathcal{E}$ satisfying the following properties:
  \begin{enumerate}
    \item
      $F(*) = 0$.
    \item
      $F$ sends short exact sequences in $\Cat^\dual_\infty$ to exact sequences in $\mathcal{E}$.
  \end{enumerate}
  We say that $F$ is \emph{finitary} if it additionally preserves filtered colimits. We write
  \[
    \Loc^\cont_\omega(\mathcal{E}) \subseteq \Loc^\cont(\mathcal{E}) \subseteq \Fun(\Cat^\dual_\infty, \mathcal{E}).
  \]
  for the full subcategories of finitary and of all continuous localizing invariants, respectively.
\end{definition}

The key tool in constructing continuous localizing invariants is the Calkin category\footnote{Specifically, we refer to the continuous $\omega_1$-Calkin category as defined in \cite{efimov2025ktheorylocalizinginvariantslarge}. Since we will only work with the continuous version, we henceforth drop that adjective and refer simply to the Calkin category.}, as defined in \cite[Definition 1.60]{efimov2025ktheorylocalizinginvariantslarge}.

\begin{definition}
  Let $\mathcal{C}$ be a dualizable stable $\infty$-category. The \emph{Calkin category} $\Calk(\mathcal{C})$ is the category of compact objects in the Verdier quotient $\Ind(\mathcal{C}^{\omega_1}) / \mathcal{C} \simeq \ker(\colim: \Ind(\mathcal{C}) \to \mathcal{C})$, i.e.,
  \[
    \Calk(\mathcal{C}) = \big( \Ind(\mathcal{C}^{\omega_1}) / \mathcal{C} \big)^{\omega}.
  \]
\end{definition}

Note that $\Calk(\mathcal{C})$ is idempotent-complete. Moreover, one has a short exact sequence\footnote{This short exact sequence can be viewed as a 2-term resolution of a dualizable stable $\infty$-category by compactly generated stable $\infty$-categories.} in $\Cat^\dual_\infty$
\begin{equation}\label{equ:calkin-localization-sequence}
  \mathcal{C} \xrightarrow{\hat{y}} \Ind(\mathcal{C}^{\omega_1}) \to \Ind(\Calk(\mathcal{C})).
\end{equation}
It follows that the Calkin construction is functorial. Indeed, given a strongly continuous functor $F: \mathcal{C} \to \mathcal{D}$ between dualizable stable $\infty$-categories, one obtains a commutative diagram
\[\begin{tikzcd}
  {\mathcal{C}} & {\Ind(\mathcal{C}^{\omega_1})} & {\Ind(\Calk(\mathcal{C}))} \\
  {\mathcal{D}} & {\Ind(\mathcal{D}^{\omega_1})} & {\Ind(\Calk(\mathcal{D}))}
  \arrow["{\hat{y}}", from=1-1, to=1-2]
  \arrow["F"', from=1-1, to=2-1]
  \arrow[from=1-2, to=1-3]
  \arrow[from=1-2, to=2-2]
  \arrow[from=1-3, to=2-3]
  \arrow["{\hat{y}}", from=2-1, to=2-2]
  \arrow[from=2-2, to=2-3]
\end{tikzcd}\]
Passing to compact objects in the rightmost column then defines the functor
\[
  \Calk(F): \Calk(\mathcal{C}) \to \Calk(\mathcal{D}).
\]
Hence, the functor
\[
  \Calk: \Cat^\dual_\infty \to \Cat^\perf_\infty
\]
is well defined.

\begin{proposition}[{\cite[Proposition 1.63, Proposition 1.73]{efimov2025ktheorylocalizinginvariantslarge}}] \label{prop:Calkin-exact}
  The Calkin construction $\Calk: \Cat^\dual_\infty \to \Cat^\perf_\infty$ preserves both short exact sequences and filtered colimits.
\end{proposition}

We are now ready to define the continuous extension of localizing invariants.

\begin{definition} \label{def:continuous-localizing-invariant-extension}
  Let $\mathcal{C}$ be a dualizable $\infty$-category and $\mathcal{E}$ a stable $\infty$-category. Given a localizing invariant $F: \Cat^\ex_\infty \to \mathcal{E}$, we define a functor $F^\cont: \Cat^\dual_\infty \to \mathcal{E}$ by
  \[
    F^\cont(\mathcal{C}) = \Omega F(\Calk(\mathcal{C})).
  \]
\end{definition}

\begin{proposition}[{\cite[Proposition 4.6, Proposition 4.8]{efimov2025ktheorylocalizinginvariantslarge}}]
  Let $\mathcal{E}$ and $F$ be as in \cref{def:continuous-localizing-invariant-extension}. The functor $F^\cont: \Cat^\dual_\infty \to \mathcal{E}$ is a localizing invariant. Moreover, if $F$ is finitary, then $F^\cont$ is also finitary.
\end{proposition}

\begin{proof}
  It follows from \cref{prop:Calkin-exact} and the fact that filtered colimit of short exact sequences in $\Cat^\dual_\infty$ is also a short exact sequence (see \cite[Proposition 1.67]{efimov2025ktheorylocalizinginvariantslarge}).
\end{proof}

\begin{proposition}[{\cite[Proposition 4.7]{efimov2025ktheorylocalizinginvariantslarge}}]
  Let $F: \Cat^\ex_\infty \to \mathcal{E}$ be a localizing invariant and $\mathcal{C}$ be a compactly generated $\infty$-category. Then
  \[
    F^\cont(\mathcal{C}) \simeq F(\mathcal{C}^\omega).
  \]
  In particular, there is an equivalence $F^\cont\circ\Ind \simeq F$ of functors $\Cat^\ex_\infty \to \mathcal{E}$.
\end{proposition}

\begin{proof}
  The short exact sequence (\ref{equ:calkin-localization-sequence}) is the $\Ind$-construction of the short exact sequence (in $\Cat^\perf$)
  \[
    \mathcal{C}^\omega \to \mathcal{C}^{\omega_1} \to \Calk(\mathcal{C}).
  \]
  Since $\mathcal{C}^{\omega_1}$ admits countable colimits, the Eilenberg swindle implies that $F(\mathcal{C}^{\omega_1})\simeq 0$. The result follows by the localizing property of $F$.
\end{proof}

\begin{example}
  We define the continuous algebraic $K$-theory $\mathrm{K}^\cont: \Cat^\dual_\infty \to \Sp$ by extending the (non-connective) algebraic $K$-theory $\mathrm{K}: \Cat^\ex_\infty \to \Sp$. In particular, for any ring $R$ one recovers the classical algebraic $K$-theory via
  \[
    \mathrm{K}(R) \simeq \mathrm{K}\bigl(D^{\perf}(R)\bigr) \simeq \mathrm{K}^{\cont}\bigl(D(R)\bigr),
  \]
  where $D^{\perf}(R)\subseteq D(R)$ denotes the full subcategory of perfect complexes, i.e., the compact objects in the derived category $D(R)$.
\end{example}

\begin{theorem}[{\cite[Theorem 4.10]{efimov2025ktheorylocalizinginvariantslarge}}] \label{thm:equivalence-loc-and-loc-cont}
  Let $\mathcal{E}$ be a stable $\infty$-category. The restriction along $\Ind : \Cat^{\ex}_\infty \to \Cat^{\dual}_\infty$ induces equivalences
  \[
    \Loc^{\cont}(\mathcal{E}) \xrightarrow{\simeq} \Loc(\mathcal{E})
  \]
  and
  \[
    \Loc^{\cont}_\omega(\mathcal{E}) \xrightarrow{\simeq} \Loc_\omega(\mathcal{E}).
  \]
  The inverse is given by the assignment $F \mapsto F^\cont$.
\end{theorem}

We proceed to study the value of a finitary continuous localizing invariant applied to a continuous six-functor formalism of interest, following the strategy of \cite[\Sec 3.6]{krause2024sheaves}.

\begin{lemma} \label{lem:dualizable-verdier-sequence}
  Let $D\in \SixFF(\LCH)^{\cont}$. Then for each closed immersion $i: Z\hookrightarrow X$ with open complement $j: U\hookrightarrow X$, the sequence
  \[
    D(U) \xrightarrow{j_!} D(X) \xrightarrow{i^*} D(Z)
  \]
  is a short exact sequence in $\Cat^{\dual}_\infty$.
\end{lemma}

\begin{proof}
  Begin with the split Verdier sequence in $\Cat^{\ex}_\infty$ (see \cref{rem:verdier-sequence-from-localization})
  \[
    D(Z) \xrightarrow{i_*} D(X) \xrightarrow{j^*} D(U),
  \]
  and pass to left adjoints to obtain
  \[
    D(U) \xrightarrow{j_!} D(X) \xrightarrow{i^*} D(Z).
  \]
  By definition, each term of this sequence is both stable and dualizable. Cofibers in $\Cat^{\dual}_\infty$ may be computed on the underlying presentable stable $\infty$-category by \cite[Proposition 1.65]{efimov2025ktheorylocalizinginvariantslarge}, and then, by passing to right adjoints, identify with fibers in the underlying stable $\infty$-categories. Moreover, the fiber of any morphism in $\Cat^{\dual}_\infty$ is the maximal dualizable full subcategory of the underlying fiber by \cite[Proposition 1.84]{efimov2025ktheorylocalizinginvariantslarge}. These observations together imply that the above sequence is exact in $\Cat^{\dual}_\infty$, as claimed.
\end{proof}

\begin{lemma} \label{lem:sheaf-of-localizing-invariant}
  Let $F\in\Loc^{\cont}_\omega(\mathcal{E})$ be a finitary continuous localizing invariant valued in a stable $\infty$-category $\mathcal{E}$, and $D\in \SixFF(\LCH)^{\cont}$ be regarded as a covariant functor via extension functoriality. Then the functor
  \[
    \LCH_{\open} \to \mathcal{E}: X \mapsto F(D(X))
  \]
  defines a cosheaf.
\end{lemma}

\begin{proof}
  By construction, the functor is reduced, i.e., it sends the empty space to a zero object, and it preserves filtered colimits. Given open subsets $U, V \subseteq X$, we obtain, by \cref{lem:dualizable-verdier-sequence}, a horizontally exact commutative diagram in $\Cat^{\dual}_\infty$
  \[\begin{tikzcd}
    {D(U\cap V)} & {D(U)} & {D(U\setminus V)} \\
    {D(V)} & {D(U\cap V)} & {D(U\setminus V)}
    \arrow[from=1-1, to=1-2]
    \arrow[from=1-1, to=2-1]
    \arrow[from=1-2, to=1-3]
    \arrow[from=1-2, to=2-2]
    \arrow["{=}", from=1-3, to=2-3]
    \arrow[from=2-1, to=2-2]
    \arrow[from=2-2, to=2-3]
  \end{tikzcd}\]
  Applying $F$, we obtain a diagram of exact sequences in $\mathcal{E}$. Since $\mathcal{E}$ is stable, it follows that the image of the left square under $F$ is a pushout square in $\mathcal{E}$, establishing the cosheaf gluing condition.
\end{proof}

\begin{lemma} \label{lem:profinite-descent-localizing-invariant}
  Let $F\in\Loc^{\cont}_\omega(\mathcal{E})$ be a finitary continuous localizing invariant valued in a stable $\infty$-category $\mathcal{E}$, and $D\in \SixFF(\LCH)^{\cont}$. Then for a cofiltered limit of compact Hausdorff spaces $X = \lim_{i \in I} X_i$ we have an equivalence
  \[
    F(D(X)) \simeq \colim_{I} F(D(X_i)).
  \]
\end{lemma}

\begin{proof}
  This follows directly from the profinite descent property of $D$ together with the finitary assumption of $F$, after passing to the corresponding left adjoint.
\end{proof}

\begin{theorem} \label{thm:localizing-invariant-sp-six-functor}
  Let $F\in\Loc^{\cont}_\omega(\Sp)$ be a finitary continuous localizing invariant, and $D\in \SixFF(\LCH)^{\cont}$. Then for any locally compact Hausdorff space $X$, there is an equivalence
  \[
    F(D(X)) \simeq \Gamma_{\mathrm{c}}\Big( X, \underline{F(D(\pt))} \Big).
  \]
\end{theorem}

\begin{proof}
  Let $X$ be a locally compact Hausdorff space. Define the functor $\mathcal{F}: \Open(X) \to \Sp$ by
  \[
    U \mapsto F(D(U)).
  \]
  by \cref{lem:sheaf-of-localizing-invariant}, this defines a cosheaf of spectra on $X$. Applying Verdier duality (\cref{thm:verdier-duality}),
  \[
    \D: \CoShv(X; \Sp) \simeq \Shv(X; \Sp),
  \]
  we obtain a sheaf $\D(\mathcal{F})$. For any compact subset $K\subseteq X$, Verdier duality gives
  \[
    \D(\mathcal{F})(K) = \mathcal{F}(X)/\mathcal{F}(X \setminus K).
  \]
  where the value of a sheaf on a subset denotes the global sections of its restriction to that subset. To describe $\D(\mathcal{F})(K)$, consider the short exact sequence
  \[
    D(U) \xrightarrow{j_!} D(X) \xrightarrow{i^*} D(Z)
  \]
  in $\Cat^{\dual}_\infty$ by \cref{lem:dualizable-verdier-sequence} in which $i: Z\hookrightarrow X$ is a closed immersion with open complement $j: U\hookrightarrow X$. Applying the continuous localizing invariant $F$ to the above short exact sequence, we have an exact sequence in $\Sp$
  \[
    \mathcal{F}(U) \to \mathcal{F}(X) \to \mathcal{F}(Z)
  \]
  Taking $U = X \setminus K$, it follows that
  \[
    \D(\mathcal{F})(K) = \mathcal{F}(X)/\mathcal{F}(X \setminus K) \simeq \mathcal{F}(K).
  \]
  Since $K$ is compact, the unique map $\pi_K: K \to \pt$ is proper and $\pi_{K}^*: D(\pt) \to D(K)$ is strongly continuous, which gives a functor
  \[
    \mathcal{F}(\pt) \to \mathcal{F}(K)
  \]
  adjoins over to a morphism $\underline{\mathcal{F}(\pt)} \to \D(\mathcal{F})|_{K}$. As $X$ is locally compact, these local morphisms glue into a morphism of sheaves $\underline{\mathcal{F}(\pt)} \to \D(\mathcal{F})$.

  The remainder of the proof follows the same argument as in the proof of \cref{lem:compact-supp-cohomology-agree-sp}. We reproduce it here for completeness. To complete the proof, we must show that the morphism $\underline{\mathcal{F}(\pt)} \to \D(\mathcal{F})$ is an equivalence. Indeed, under Verdier duality, a sheaf $\mathcal{G}$ corresponds to the cosheaf $\D^{-1}(\mathcal{G})$ whose value on $U$ is the compactly supported sections of $\mathcal{G}$ on $U$, which then gives the desired description of $\mathcal{F}$.

  It remains to verify that $\underline{\mathcal{F}(\pt)} \to \D(\mathcal{F})$ is an equivalence. As sheaves on locally compact spaces are determined by their values on compact subsets, and constant sheaves are stable under pullback, it suffices to check the equivalence when $X$ is compact. Moreover, if the equivalence holds on a compact space, it holds on each of its closed subspaces. Now every compact Hausdorff space embeds into a Hilbert cube $[0,1]^I$ for some set $I$, via the canonical double dual map $X \to [0,1]^I$, where $I = C^0(X, [0,1])$. This map is injective by Urysohn’s Lemma and defines an embedding. Thus, it suffices to prove the claim for $[0,1]^I$.

  The Hilbert cube is an inverse limit of finite-dimensional cubes $[0,1]^n$. By \cref{prop:section-profinite-descent} together with \cref{lem:profinite-descent-localizing-invariant}, it is enough to verify the claim for each $[0,1]^n$. These spaces are hypercomplete, the equivalence of sheaves can be checked on stalks. Since both $\underline{\mathcal{F}(\pt)}$ and $\D(\mathcal{F})$ have stalks equivalent to $\underline{\mathcal{F}(\pt)}$, the morphism is indeed an equivalence.
\end{proof}

\begin{theorem} \label{thm:localizing-invariant-dualizable-six-functor}
  Let $F\in\Loc^{\cont}_\omega(\mathcal{E})$ be a finitary continuous localizing invariant valued in a dualizable stable $\infty$-category $\mathcal{E}$, and $D\in \SixFF(\LCH)^{\cont}$. Then for any locally compact Hausdorff space $X$, there is an equivalence
  \[
    F(D(X)) \simeq \Gamma_{\mathrm{c}}\Big( X, \underline{F(D(\pt))} \Big).
  \]
\end{theorem}

\begin{proof}
  The argument proceeds as in the proof of \cref{thm:localizing-invariant-sp-six-functor}, now working with $\mathcal{E}$-valued sheaves. The required stalkwise equivalence is then verified using \cref{lem:stalkwise-check}.
\end{proof}

As noted in \cite[Remark 3.6.17]{krause2024sheaves}, the dualizability assumption on $\mathcal{E}$ in \cref{thm:localizing-invariant-dualizable-six-functor} can be relaxed. It suffices to assume that $\mathcal{E}$ is presentable. This follows from the existence of a universal finitary localizing invariant
\[
  \mathcal{U}_{\mathrm{loc}}: \Cat^{\ex}_\infty \to \mathcal{M}_{\mathrm{loc}}
\]
where $\mathcal{M}_{\mathrm{loc}}$ is the stable $\infty$-category of noncommutative motives constructed by Blumberg–Gepner–Tabuada \cite{Blumberg_2013}. By the universal property of $\mathcal{U}_{\mathrm{loc}}$ (see \cite[Theorem 1.1]{Blumberg_2013}), any finitary localizing invariant $F: \Cat^{\ex}_\infty \to \mathcal{E}$, with $\mathcal{E}$ a presentable stable $\infty$-category, factors through $\mathcal{M}_{\mathrm{loc}}$ via a unique colimit-preserving functor $\tilde{F}: \mathcal{U}_{\mathrm{loc}} \to \mathcal{E}$. Moreover, by \cref{thm:equivalence-loc-and-loc-cont}, there exists a universal finitary continuous localizing invariant such that any continuous finitary localizing invariant factor through $\mathcal{M}_{\mathrm{loc}}$.
\[
\begin{tikzcd}
  {\Cat^{\ex}_\infty} & {\mathcal{E}} \\
  {\mathcal{M}_{\mathrm{loc}}}
  \arrow["F", from=1-1, to=1-2]
  \arrow["{\mathcal{U}_{\mathrm{loc}}}"', from=1-1, to=2-1]
  \arrow["{\tilde{F}}"', from=2-1, to=1-2]
\end{tikzcd}
\qquad
\begin{tikzcd}
  {\Cat^{\dual}_\infty} & {\mathcal{E}} \\
  {\mathcal{M}_{\mathrm{loc}}}
  \arrow["{F^{\cont}}", from=1-1, to=1-2]
  \arrow["{\mathcal{U}^{\cont}_{\mathrm{loc}}}"', from=1-1, to=2-1]
  \arrow["{\tilde{F}^{\cont}}"', from=2-1, to=1-2]
\end{tikzcd}
\]
Therefore, it suffices to consider the universal case, and the dualizability assumption need only be verified for $\mathcal{M}_{\mathrm{loc}}$.

\begin{theorem}[{Efimov, \cite[Remark 4.3]{efimov2025ktheorylocalizinginvariantslarge}\cite[Theorem 4.7.1]{krause2024sheaves}}]
  The stable $\infty$-category $\mathcal{M}_{\mathrm{loc}}$ is dualizable.
\end{theorem}

Combining these results, we obtain the following generalization:

\begin{theorem} \label{thm:localizing-invariant-six-functor}
  Let $F\in\Loc^{\cont}_\omega(\mathcal{E})$ be a finitary continuous localizing invariant valued in a presentable stable $\infty$-category $\mathcal{E}$, and $D\in \SixFF(\LCH)^{\cont}$. For any locally compact Hausdorff space $X$, there is an equivalence
  \[
    F^{\cont}(D(X)) \simeq \Gamma_{\mathrm{c}}\Big(X, \underline{F^{\cont}(D(\pt))}\Big).
  \]
\end{theorem}

\begin{corollary}[Efimov]
  For any locally compact Hausdorff space $X$ there is an equivalence
  \[
    \mathrm{K}^{\cont}(\Shv(X; \mathcal{C})) \simeq \Gamma_{\mathrm{c}}\Big(X, \underline{\mathrm{K}^{\cont}(\mathcal{C})}\Big).
  \]
\end{corollary}

\begin{definition}[Motivic equivalence]
  A functor $F : \mathcal{C} \to \mathcal{D}$ in $\Cat^{\dual}_\infty$ is called a \emph{motivic equivalence} if the induced morphism
  \[
    \mathcal{U}_{\mathrm{loc}}^{\cont}\, \mathcal{C} \to \mathcal{U}_{\mathrm{loc}}^{\cont}\, \mathcal{D}
  \]
  is an equivalence in the $\infty$-category of noncommutative motives $\mathcal{M}_{\mathrm{loc}}$.
\end{definition}

\begin{corollary}
  For any locally compact Hausdorff space $X$ such that the unit $1_{D(X)}$ is compact in $D(X)$, the component at $X$ of the canonical morphism of continuous six-functor formalisms
  \[
    \Shv(-; D(\pt)) \to D
  \]
  is a motivic equivalence.
\end{corollary}

\begin{proof}
  The right adjoint of the canonical morphism $\Shv(-; D(\pt)) \to D$ is given by a functor $D \to \Shv(-; D(\pt))$ which assigns to each object $d\in D(X)$ the presentable sheaf
  \[
    \big( U \xhookrightarrow{i} X \big) \,\mapsto\, \underline{\Hom}_{D(X)}(i_! i^* 1_{D(X)}, d).
  \]
  By the projection formula, we have
  \[
    \underline{\Hom}_{D(X)}(i_! i^* 1_{D(X)}, d) \,\simeq\, \underline{\Hom}_{D(X)}(1_{D(X)}, i_* i^* d) \,\simeq\, \underline{\Hom}_{D(X)}(1_{D(X)}, d \otimes i_* 1_{D(U)}).
  \]
  Since $1_{D(X)}$ is compact the right adjoint preserves filtered colimits and hence all colimits. It follows that the canonical morphism $\Shv(-; D(\pt)) \to D$ is strongly continuous. The desired motivic equivalence then follows immediately from \cref{thm:localizing-invariant-six-functor}.
\end{proof}

\begin{remark}
  In fact, \cref{thm:localizing-invariant-six-functor} yields a slightly stronger statement. For any continuous six-functor formalism $D$, there is a natural equivalence
  \[
    \mathcal{U}_{\mathrm{loc}}^{\cont}\, \Shv(-; D(\pt)) \,\simeq\, \mathcal{U}_{\mathrm{loc}}^{\cont}\, D
  \]
   In particular, the canonical morphism $\Shv(-; D(\pt)) \to D$ may be regarded as a motivic equivalence of continuous six-functor formalisms.
\end{remark} 
\section{Outlook}

In the proof of the universal property of sheaves (\cref{thm:sheaf-initial-six-functor}), we apply profinite descent only to cofiltered limits of intervals $[0,1]^n$, and hyperdescent is used only for these spaces. Consequently, the universal property of sheaves extends to a slightly larger $\infty$-category of six-functor formalisms.

By contrast, in motivic homotopy theory, \cite[Proposition 5.21]{drew2018motivichodgemodules} shows that morphisms between coefficient systems are automatically compatible with the right adjoint $f_*$ for proper maps $f$, without requiring extra assumptions. Therefore, the universal property of the initial object in the $\infty$-category of coefficient systems lifts directly to the $\infty$-category of Nagata six-functor formalisms. This naturally leads to the following question:

\begin{question}
  Is $\Shv(-; \Sp)$ initial among all Nagata six-functor formalisms on locally compact Hausdorff spaces that satisfy only the localization property?
\end{question}

Once the universal property of sheaves has been established, it follows that for any Nagata six-functor formalism $D$, the canonical morphism $\Shv(-; D(\pt)) \to D$ is a cohomological equivalence. This raises the subsequent question:

\begin{question}
  Is the canonical map $\Shv(-; D(\pt)) \to D$ an equivalence? If not, do there exist nontrivial Nagata six-functor formalisms distinct from the sheaf, and how can one classify all Nagata six-functor formalisms?
\end{question} 

\bibliographystyle{hyperalpha}
\bibliography{references}

\end{document}